\documentclass[a4paper,notitlepage,twoside,reqno,11pt]{amsart}

\usepackage{mathrsfs}
\usepackage{yhmath}
\usepackage{tikz}
\usepackage[hmargin=1in,vmargin=1in]{geometry}
\usepackage{amsmath}
\usepackage{amssymb}
\usepackage{amsfonts}
\usepackage{amsthm}
\usepackage{bbm,pifont,latexsym}
\usepackage{anysize}  
\marginsize{3cm}{3cm}{3cm}{3cm}

\usepackage{dcolumn,indentfirst,color,hyperref}
\usepackage{amsmath,amssymb,amscd,amsthm,amsfonts,mathrsfs}
\usepackage{color,graphicx,xcolor,graphics}
\usepackage{titlesec} 
\usepackage{titletoc} 
\titlecontents{section}[20pt]{\addvspace{2pt}\filright}
              {\contentspush{\thecontentslabel. }}
              {}{\titlerule*[8pt]{}\contentspage}

\usepackage{fancyhdr} 
\def\mod{\operatorname{mod}}
\newtheorem{thm}{Theorem}[section]

\newtheorem{lem}[thm]{Lemma}
\newtheorem{cor}[thm]{Corollary}
\newtheorem{prop}[thm]{Proposition}

\newtheorem{defi}{Definition}[section]
\newtheorem{conj}{Conjecture}[section]

\newtheorem{notation}[thm]{Notation}

\newtheorem{proposition}[thm]{Proposition}

\newcommand{\wt}{\widetilde}
\newcommand{\mb}{\mathbb}
\newcommand{\mc}{\mathcal}
\newcommand{\sm}{\setminus}
\newcommand{\tu}{\textup}
\newcommand{\ol}{\overline}
\newcommand{\es}{\emptyset}

\newcommand{\tb}{\textbf}

\newcommand{\wh}{\widehat}
\newcommand{\olC}{\wh{\mathbb{C}}}

\def\AA{\mathcal{A}}\def\BB{\mathcal{B}}

\makeatletter\@addtoreset{equation}{section}\makeatother 
\titleformat{\section}{\centering\normalsize}{\textsc{\thesection.}}{1em}{\textsc}
\titleformat{\subsection}{\normalsize}{\thesubsection.}{1em}{\textbf}

\title
 {Rigidity of non-renormalizable Newton maps}
\begin{document}

\author{Pascale Roesch}
\address{Pascale Roesch, Institut de Math\'ematiques de Toulouse, Universit\'e Paul Sabatier, 118, Route de Narbonne, 31062 Toulouse Cedex, France}
\email{pascale.roesch@math.univ-toulouse.fr}

\author{Yongcheng Yin}
\address{Yongcheng Yin, School of Mathematical Sciences, Zhejiang University, Hangzhou, 310027, P.R. China}
\email{yin@zju.edu.cn}

\author{Jinsong Zeng}
\address{Jinsong Zeng, School of Mathematics and Information Science, Guangzhou University, Guangzhou 510006, P. R. China}
\email{jinsongzeng@163.com}

\begin{abstract}
Non-renormalizable Newton maps are rigid. More precisely, we prove that their Julia sets carry no invariant line fields and that a topological conjugacy between them is equivalent to a quasiconformal conjugacy.
\end{abstract}



\maketitle
\section{Introduction}

In this article we consider the question of rigidity of rational maps in the context of the  family of Newton maps. These maps are of the form $$f_p(z)=z-\frac{p(z)}{p'(z)},$$ where $p:\mb{C}\to\mb{C}$ is a polynomial. They are  well known because 
their iterations are used to find roots of the polynomial $p$. Recently the dynamics of Newton maps has been studied a lot, see \cite{AR,BFJK, DLSS,DMRS,LMS1,LMS2,Roe,RWY,WYZ} etc.

The question of rigidity is to know whether two maps are conformally conjugated when they are topologically or quasiconformally conjugated. This problem is natural and important, since it ``says" that understanding the behavior of orbits prescribes the map.

In the context of rational maps of degree $d\geq 2$, it is stated as the NILF Conjecture: 
\begin{conj}[NILF] The Julia set of a rational map does not support an invariant line field except it is a Latt$\grave{e}$s example.
\end{conj}
A rational map $f:\olC\to\olC$ admits an \emph{invariant line field supported on its Julia set}
$\mc{J}_f$ provided that there is a measurable Beltrami differential $\mu
(z)\frac{d\bar{z}}{dz}$ supported on a positive measure subset of $\mc{J}_f$ with $|\mu|=1$ and $f^*\mu=\mu$ a.e..

The  NILF Conjecture  is central since it implies the density of hyperbolic maps in the space
of all rational maps of degree $d$; see \cite{McS}. The NILF
Conjecture has been studied by many people; see
\cite{GS,Hai,KS,Levin-vanStrien,McMbook,Sh08,Shen,Yoccoz}.
But it is still open even for quadratic polynomials.

Our main results here concern  rational maps of the previous type: Newton maps. 

\begin{thm}\label{th:NILF}
	The Julia set of a non-renormalizable Newton map carries no invariant line fields.
\end{thm}

\begin{thm}\label{th:conjugate}
	If two non-renormalizable Newton maps are topologically conjugated, then they are quasiconformally conjugated.
\end{thm}

A Newton map $f$ is said to be {\it renormalziable} if there exists an integer $k\geq 1$ and two topological disks $U\Subset V \subset \mathbb{C}$ such that $g:=f^k:U\to V$ is a polynomial-like map with a connected {\it filled-in Julia set} $K(g):=\bigcap_{n\geq 0}g^{-n}(V)$.

There are two remarks. (i) The dynamic of renormalizable part of Newton maps is based on a polynomial and the rigidity of polynomials is a difficult question; see \cite{McMbook} and others. (ii) One of the main challenges for generalizing Theorems \ref{th:NILF} and \ref{th:conjugate} to other families is the construction of puzzles. For Newton map, the key propery that it has just one non-attracting fixed point enables us to construct invariant graphs and puzzles.

The proof of Theorem \ref{th:NILF} is sketched as follows.
Assume by contradiction that $f$ admits  an  invariant line field $\mu$. The invariant graphs $G$ constructed in \cite{WYZ} for Newton maps of arbitrary degrees allow us to build the puzzles. Then by means of distortion properties of holomorphic maps as well as Kozlovski-Shen-vanStrien's \emph{enhanced nests} and Kahn-Lyubich's covering Lemma, we may give a good control of the puzzle geometry. Finally, by a criterion in \cite{Shen}, we can prove that  either $\mu(z)= 0 $ or $\mu$ is not almost continuous at a point $z\in\mc{J}_f$. This  finally implies that $\mu=0$. 
 
Note that in \cite{DS} a similar result can be found in Theorem B which is essentially the same  as Theorems \ref{th:NILF} and \ref{th:conjugate} here but stated in another form.  These results are developed independently. The first version of our result can be found on the  preprint (arXiv:1811.09978),
 K.Drach and D.Schleicher gave another proof of Theorems 1.1 and 1.2 in a preprint (arXiv:1812.11919). 
 One can see two major differences.
 
 First,  one of the key steps is to prove that the end of a wandering point is trivial when its $\omega$-limit set intersects a periodic end. Their approach goes along the proof of \cite[Proposition 5.4]{RY} by the first and second authors. This proof is very technical. For Newton maps, we use the dynamics at infinity to give a very simple proof; see Lemmas \ref{lem:elevator} and \ref{lem:preperiodic}. Moreover, the arguments can easily be adapted to a more general situation.
 
Another difference is lying in the construction of puzzles. In the present paper, the graph used has very good properties: it has no cut points except some strictly pre-periodic Fatou centers; thus the construction of puzzles is natural and very simple (see Figure \ref{fig:cut}); in particular, these puzzles are defined for arbitrary iterates of Newton map. In contrast, the construction of puzzles used in \cite{DS} is more tedious: the graph is a component of high-iterated pre-image of a channel diagram (they called it Newton graph); in order to get  puzzle pieces that are  Jordan disks, it is then another difficulty in \cite{DS} to find some ``nice" circles in Newton graph, and to pull back the union of the channel diagram and these circles. The puzzles obtained in this way make sense only for a suitable high iterate of Newton map.

The paper is organized as follows.
In Section \ref{sec:distortion}, we present some distortion lemmas for holomorphic maps which are used in Section \ref{sec:proofs}. In Section \ref{sec:Newton} we recall some basic facts about the dynamics of Newton maps. In Section \ref{sec:puzzles} we give a precise construction of puzzles from the invariant graph obtained in \cite{WYZ}. Section \ref{sec:wandering} deals with wandering points having nice properties, i.e., satisfying \emph{the elevator condition}. In Section \ref{sec:recurrent}, we recall the properties of Kozlovski-Shen-vanStrien's enhanced nests associated to persistently recurrent critical points. 
In Section \ref{sec:proofs}, we give the proofs of Theorems \ref{th:NILF} and \ref{th:conjugate}.

\begin{notation}
	We will use the following notations frequently.

\begin{itemize}
	\item  $\widehat{\mathbb{C}}$, $\mathbb{C}$, $\mathbb{D}$ are the Riemann sphere, the plane and  the unit disk respectively.
	\item The open disk centered at $a\in \mathbb{C}$ with radius $r>0$ is denoted by $D(a,r)$.
	\item Let $A$ be a non-empty set in $\widehat{\mathbb{C}}$.  We write $\#A$ for the cardinality of $A$. The closure and the boundary of $A$ are denoted by $\overline{A}$ and  $\partial A$ respectively.  
	We denote by $\rm{Comp}(A)$  the collection of all connected components of $A$. The diameter of $A$ measured in the Euclidean metric is denoted by $\tu{diam}(E)$.
	\item For two subsets $A$ and $B$ of $\widehat{\mathbb{C}}$, the notation   $A\Subset B$  means that $\overline{A}$ is contained in the interior of $B$.
	\item Let $f$ be a rational map on $\olC$. We write $\tu{crit}(f)$ as the set of critical points of $f$. The Julia set of $f$ is denoted by $\mc{J}_f$. The orbit of $z\in\olC$ under $f$ is $\tu{orb}(z):=\bigcup_{n\geq 0}f^n(z)$.
\end{itemize}
\end{notation}

\vskip 1 em \noindent{\it Acknowledgment}: 
The authors would like very much to thank Professor Xiaoguang Wang for helpful discussion. This research is partially supported by the ANR ABC and the NSF of China.

\section{Some distortion lemmas}\label{sec:distortion}
This section collects some distortion properties for proper holomorphic maps, which will be used later on.

Let $A\subset\mb{C}$ be an annulus with non-degenerated boundary components. Then there exists a confromal map sending $A$ to a standard annulus $\{z\in\mb{C}: 0<r<|z|<1\}$, where $r>0$ is uniquely determined by $A$. The \emph{modulus} of $A$ is defined as $\tu{mod}(A)=\frac{1}{2\pi}\tu{log}(1/r)$, which is invariant under conformal maps. We need the following distortion property of modulus under proper holomorphic maps.

\begin{lem}\cite[Lemma 4.5]{KL1}\label{lem:moduli}
	Let $U_i\Subset V_i\subset \mb{C}$ be topological disks, where $i=1,2$. Suppose that $f:V_1\to V_2$ is a proper holomorphic map of degree $\delta\geq 1$ and $U_1$ is a component of $f^{-1}(U_2)$. Then 
	$$\tu{mod}(V_1\setminus\ol{U}_1)\leq\tu{mod}(V_2\setminus\ol{U}_2)\leq \delta\cdot \tu{mod}(V_1\setminus\ol{U}_1).$$
	
\end{lem}
Consider a bounded topological disk $U$ in $\mb{C}$ and a point $z\in U$. The \emph{shape} of $U$ about $z$ is defined by
$${\rm Shape}(U,z)=\frac{\sup_{x\in \partial U}|x-z|}{\inf_{x\in \partial U}|x-z|}
.$$

It is clear that ${\rm Shape} (U,z)=1$ if and only if $U$ is a round
disk centered at the point $z$.

\begin{lem}\label{4:distortionlemma}
	Let $U\Subset V$ be bounded topological disks in
	$\mathbb{C}$ with  $\tu{mod}(V\setminus\overline{U})\ge m>0$. Then there is a constant $C=C(m)>0$
	such that for any two points $x,y\in U$,
	$${\rm Shape} (V,y)\leq C\cdot{\rm Shape} (V,x).$$
\end{lem}
\begin{proof}
	By \cite[Lemma 1]{YZ}, there is a constant $c=c(m)$ such that for any $z\in U$ we have $$\tu{diam}(U)\leq c\cdot \tu{min}_{w\in\partial V}|z-w|.$$  Let $R_z:=\tu{max}_{w\in\partial V}|z-w|$ and $r_z:=\tu{min}_{w\in\partial V}|z-w|$. We estimate that
	$$R_y\leq |x-y|+R_x\leq\tu{diam}(U)+R_x\leq c\cdot r_x+R_x\leq (c+1)R_x.$$ On the other hand, we assume $r_y=|y-w|$ for some $w\in\partial V$ and it then holds that
	$$r_x\leq |x-w|\leq |x-y|+|y-w|\leq \tu{diam}(U)+r_y\leq (c+1)r_y.$$
	The above two inequalities imply that the lemma follows with $C:=(c+1)^2.$
\end{proof}

The next two lemmas tell us that shapes and diameters of disks obtained by pullback via proper holomorphic maps can be controlled by some constants depending only on degrees and moduli.

\begin{lem}\cite[Lemma 6.1]{QWY}\label{2:distortionlemma}
	Let $E_i\subseteq U_i\Subset V_i$, $i\in \{1,2\}$, be topological disks in
	$\mathbb{C}$ with $\tu{mod}(V_2\setminus U_2)\geq m>0$. Suppose that $g:V_1\to V_2$ is
	a proper holomorphic map of degree $\leq\delta$, $E_1$ and $U_1$ be
	components of $g^{-1}(E_2)$ and $g^{-1}(U_2)$, respectively. Then, there is a constant $C=C(\delta, m)$ such that
	\begin{enumerate}
		\item for all $z\in E_1$, the shape satisfies
		$${\rm Shape} (E_1,z)\leq C\cdot{\rm Shape} (E_2,g(z));$$
		\item $\tu{diam}(U_1)/\tu{diam}(E_1)\leq C\cdot \tu{diam}(U_2)/\tu{diam}(E_2).$
	\end{enumerate}
\end{lem}

\begin{lem}\cite[Lemma 2]{YZ}\label{3:distortionlemma}
	Let $0\in U_i\Subset V_i\Subset \mb{D}$, $i\in\{1,2\}$, be topological disks. Suppose that $g: (U_1, V_1, \mb{D})\to(U_2, V_2, \mb{D})$ is a holomorphic map with $\deg(g|_{U_1})=\deg(g |_{V_1}) = \deg(g
	|_{\mathbb{D}}) = \delta\geq 2$. If $ \mod(V_2 \setminus U_2) \ge m
	> 0$, then there exists a constant $C=C(\delta,m) > 0$ such that $${\rm Shape}(V_1, 0) \le
	C\cdot{\rm Shape}(V_2, 0)^\frac1\delta.$$
\end{lem}

\section{Some basic facts for Newton maps}\label{sec:Newton}
Let $p:\mb{C}\to\mb{C}$ be a complex polynomial with at least two distinct roots. It can be factored as
$$p(z)=a(z-\xi_1)^{n_1}\cdots(z-\xi_d)^{n_d},$$
where $a\neq 0$ and $\xi_1,\cdots, \xi_d\in\mb{C}$ are  distinct roots of $p$, with multiplicities $n_1,\cdots,n_d\geq 1$, respectively.
Its Newton map $f_p:\olC\to\olC$, defined as $f_p(z)=z-\frac{p(z)}{p'(z)}$, fixes each root $\xi_k$ with multiplier
$$f_p'(\xi_k)=\frac{p(z)p''(z)}{p'(z)^2}\Big|_{z=\xi_k}=\frac{n_k-1}{n_k}.$$

Therefore, each root $\xi_k$ of $p$ corresponds to  an   attracting  fixed point of $f_p$ with multiplier $\frac{n_k-1}{n_k}$.
It follows from the equation
$$\frac{1}{f_p(z)-z}=-\sum_{k=1}^d\frac{n_k}{z-\xi_k}$$
that the degree of $f_p$ equals $d$, the number of  distinct roots of $p$.
One also verifies that $\infty$ is a repelling fixed point of $f_p$ with multiplier
$$\lambda_\infty:=\frac{\sum_{k=1}^d n_k}{\sum_{k=1}^d n_k-1}.$$

From the above discussion, we see that a degree-$d$ Newton map has $d+1$ distinct fixed points in $\olC$ with specific multipliers. Conversely, a well-known result of Head states that the positions of fixed points together with specific multipliers can determine a unique Newton map:

\begin{thm}\cite{He87}\label{char} A rational map $f:\olC\rightarrow \mathbb{\widehat{C}}$ of degree $d\geq2$
	is the Newton map of a polynomial $p$ if and only if
	$f(\infty)=\infty$ and for all other fixed points $\xi_1, \cdots, \xi_d\in \mathbb{C}$, there are  integers $n_k\geq 1$ so that
	$f'(\xi_k)=\frac{n_k-1}{n_k}$ for all $1\leq k\leq d$. In this case, $p$ has the form $a(z-\xi_1)^{n_1}\cdots(z-\xi_d)^{n_d}$ with $a\neq0$.
\end{thm}
A Newton map of degree $d=2$ has two completely invariant (super-)attracting Fatou domains and its Julia set is a quasi-circle. Theorem \ref{th:NILF} and Theorem \ref{th:conjugate} are obviously true. In the rest of this article, it is always assumed that $d\geq 3$.

Now, for the Newton map $f=f_p$ of $p$. Let $B_f$ be the set of points in $\olC$ whose orbits under $f$ converge to the roots of $p$. Then $B_f$ is open and completely $f$-invariant. The component of $B_f$ containing the root $\xi_i$, denoted by $B_i$, is called the \emph{immediate root basin} of the fixed point $\xi_i$. Clearly $f(B_i)=B_i$ and $B_f=\bigcup_{k=0}^{+\infty}f^{-k}(B_1\cup\cdots\cup B_d)$.

A Newton map $f$ is called \emph{postcritically finite} on ${B}_f$,  if there are only finitely many postcritical points in $B_f$, or equivalently, each critical point in ${B}_f$ will eventually hit one of $\xi_1,\cdots,\xi_d$. Let $\mathcal{N}_d$ be the collection of all degree-$d$ Newton maps $f$, and let
$$\mathcal{N}^{*}_d:=\{f\in \mathcal{N}_d: f \text{ is postcritically finite on } {B}_{f}.\}.$$

Each element $f\in\mc{N}_d$ determines a unique map $f^*$ in $\mc{N}_d^*$ up to affine conjugacy. Indeed, according to \cite{Sh09}, the Julia sets of Newton maps are always connected, or equivalently, all Fatou domains are simply connected. We take a topological disk $\wt{B}_i$ in each immediate root basin $B_i$ such that $f(\wt{B}_i)\Subset \wt{B}_i$ and $\wt{B}_i$ covers all critical orbits in $B_i$. Then there exists an integer $N$ such that the set
$$\wt{B}_f:=f^{-N}(\bigcup_{i=1}^d\wt{B}_i)\Subset B_f$$
formed by topological disks contains all the critical orbits in $B_f$. One can apply a standard quasiconformal surgery on $B_f$ (see \cite{BF14})
 to obtain a quasiconformal map $h:\olC\to \olC$ and a rational map $f^*$, such that 
 
 \begin{itemize}
 	\item $h({\xi}_i),1\leq i\leq d,$ are the super-attracting fixed points of $f^*$;
 	\item $f^*$ is postcritically finite on $h(B_f)$;
 	\item $h\circ f=f^*\circ h$ on $\widehat{\mathbb{C}}\setminus \wt{B}_f$;
 	\item $(f^*)'(\infty)=\frac{d}{d-1}(>1)$ by \cite[Theorem 12.4]{M1}.
  \end{itemize}
Theorem \ref{char} implies that $f^*\in \mathcal{N}^{*}_d$. Since $\ol{\partial}h=0$ on $\olC\setminus B_f$, the map $f^*$ is unique up to a affine conjugacy.
 
The virtue of introducing the family $\mc{N}_d^*$ is that one can give a natural dynamical parameterization of the set $B_{f^*}$; see \cite{M1}.

\begin{lem}
	\label{lem:coordinates}
	Let $f^*\in\mc{N}_d^*$. Then there exist, so-called \emph{B\"ottcher coordinates} $$\{(B,\Phi_B)\}_{B\in{\rm{Comp}} (B_{f^*})},$$ such that
	\begin{itemize}
		\item[(1)] each $\Phi_B:B\to \mathbb{D}$ is conformal;
		\item[(2)] $\Phi_{f^*(B)}\circ f^*\circ\Phi_B^{-1}(z)=z^{d_{B}}, z\in\mathbb{D}$, where $d_B={\rm deg}(f^*|_{B})$.
	\end{itemize}
\end{lem}

In general, for each $B\in {\rm Comp}({B}_{f^*})$, the B\"ottcher map $\Phi_B$ is not unique. There are $d_B-1$ choices of $\Phi_B$ when $f^*(B)=B$, and
$d_B$ choices of $\Phi_B$ when $f^*(B)\neq B$ and $\Phi_{f^*(B)}$ is determined.
Once the B\"ottcher coordinates $\{(B,\Phi_B)\}_{B\in{\rm{Comp}}({B}_{f^*})}$ is fixed,  we may introduce the notion of {\it internal rays}.

Let $B$ be a component of ${B}_{f^*}$,  the point $\Phi_B^{-1}(0)$ is called the \emph{center} of $B$, and the Jordan arc
$$R_B(\theta):=\Phi_B^{-1}(\{re^{2\pi i\theta}:0\leq r<1\})$$
is called the \emph{internal ray} of angle $\theta$ in $B$. According to Douady-Hubbard's theory, when $\theta$ is rational, the ray $R_B(\theta)$ always \emph{lands}, i.e., it accumulates at only one point in $\partial B$.
\section{Branner-Hubbard-Yoccoz puzzles}\label{sec:puzzles}
By a \emph{graph} we mean a connected and finite graph in $\olC$. Precisely, a graph is a connected set in $\olC$, which can be written as the union of finitely many closed arcs with their interiors pairwise disjoint. A point $z$ in a graph $G$ is called a \emph{cut point} (resp. \emph{non-cut point}) with respect to $G$ if the set $ G\setminus\{z\}$ is disconnected (resp. connected). 
It is useful to observe that all components of $\olC\setminus G$ are Jordan domains if and only if $G$ has no cut points. 

Let $G$ be a graph with finitely many cut points. There is a natural way to produce a new graph from $G$ such that it has no cut points; see Figure \ref{fig:cut}.

\begin{figure}[h]
	\begin{tikzpicture}
	\node at (0,0) {\includegraphics[width=9cm]{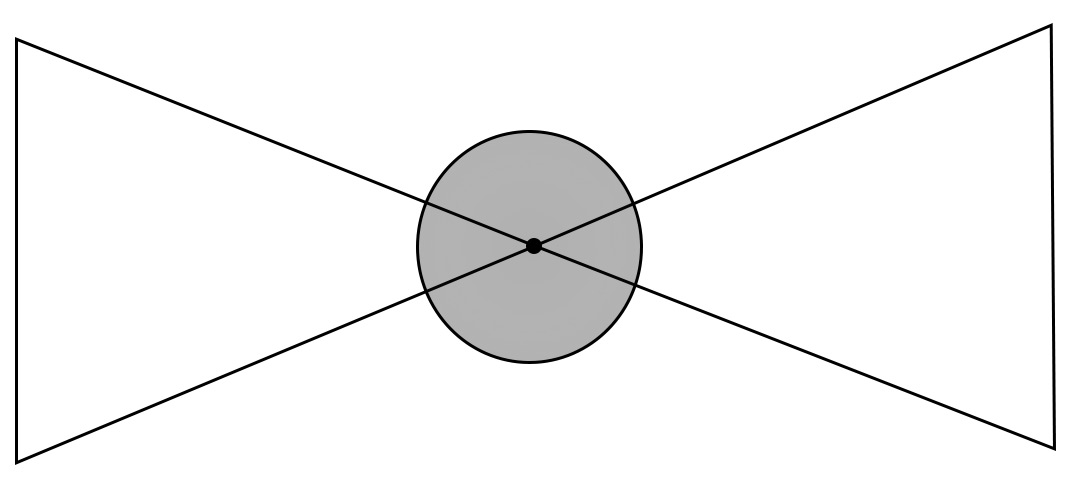}};
	\node at (0.75,0.7) {$D_z$};
	\node at (0.15,-0.29) {$z$};
	\node at (3,0) {$G$};
	\end{tikzpicture}
	\caption{The graph $G$ is the union of boundaries of the two symmetric triangles as shown above. Since the center $z$ is a cut point, the unbounded component of $\olC\setminus G$ is not a Jordan domain. We take a closed topological disk $D_z$ around the point $z$ such that $D_z\cap G$ is a star-like tree. The new graph $\partial(G\cup D_z)\,(=(G\setminus D_z)\cup\partial D_z)$ has no cut points. In general, if $G$ possesses finitely many cut points $z_1,\ldots, z_n$, then the new graph $\partial (G\cup\bigcup_{1\leq i\leq n}D_{z_i})$ has no cut points, where $D_{z_i}$ are pairwise disjoint closed topological disks associated to $z_i$. This is cited from \cite[Figure 15]{WYZ}}
	\label{fig:cut}
\end{figure}

\subsection{Puzzle pieces for maps in $\mc{N}_d^*$ and $\mc{N}_d$}
For each $f^*\in \mc{N}^*_d$, there is a natural $f^*$-invariant graph, called \emph{channel graph}, associated to $f^*$. Indeed, for any immediate root basin $B$ of $f^*$, there are exactly $d_B-1$ fixed internal rays in $B$:
$$R_B(j/(d_B-1)), 0\leq j\leq d_B-2,\tu{ where }d_B=\tu{deg}(f|_B).$$ Each of these rays must land at a fixed point on $\partial B$. Since $\infty$ is the unique fixed point of $f$ on its Julia set, all fixed internal rays land the the common point $\infty$. The \emph{channel graph} of $f^*$, denoted by $\Delta$, is defined by 
$$\Delta:=\bigcup_{B}\bigcup_{j=0}^{d_B-2}\ol{R_B(j/(d_B-1))},$$ where $B$ ranges over all immediate root basins $B_1, \cdots, B_{d}$. Clearly $f^*(\Delta)=\Delta$. 

In general, the point $\infty$ is a cut point with respect to $\Delta$ and the preimage $(f^*)^{-1}(\Delta)$ is disconnected. However, in \cite{WYZ} the 
following invariant graph with nice properties is constructed. It is the key point to construct the Branner-Hubbard-Yoccoz puzzles for Newton maps.

\begin{prop}\cite[Theorem 4.1 and Proposition 4.8]{WYZ}\label{prop:invariant}
	Let $f^*\in \mathcal{N}^{*}_d$. There exists a graph $G\subset\olC$ such that
	\begin{enumerate}
		\item $f^*(G)\subseteq G$ and $ (f^*)^{-1}(G)$ is connected;
		\item $(f^*)^k(G)=\Delta$ for some $k\geq 0$;
		\item The only possible cut points in $G$ are the centers of components of $B_{f^*}$.
	\end{enumerate}
\end{prop}
There are two remarks. First, the graph $G$ in Proposition \ref{prop:invariant} can be decomposed into two disjoint parts: $G\cap\mc{J}_{f^*}$ and $G\cap B_{f^*}$. The former part intersects $\mc{J}_{f^*}$ in finitely many points, all of which are iterated preimages of $\infty$, while the latter part, belonging to the union of finitely many components of $B_f$, is eventually mapped onto $\Delta\setminus\{\infty\}$; 
Second, each $G_n:=(f^*)^{-n}(G)$ for $n\geq 1$ still satisfies these three conditions. Indeed, it is easy to see that conditions (1) and (2) hold for $G_n$, while condition (3) follows by Proposition \ref{prop:Jordan}.

The construction of puzzle pieces for $f^*$ is as follows.  Let $G_0$ be the graph in Proposition \ref{prop:invariant}. 
Let 
$$\mc{D}_0=\bigcup_{B\cap G\neq \es, B\in\tu{Comp}(B_{f^*})}\ol{\Phi_B^{-1}(D(0,1/2))}$$
formed by pairwise disjoint closed topological disks. Then $\mc{D}_{n}:=(f^*)^{-n}(\mc{D}_0)\subset \mc{D}_{n+1}$ for any $n\geq 0$ by Lemma \ref{lem:coordinates} (2). Let $\wt{G}_n:=G_n\cup \mc{D}_n$. An element in $\mc{P}^*_n:=\tu{Comp}(\olC\setminus \wt{G}_n)$ is called a \emph{puzzle piece} of depth $n$ for $f^*$.

Since $\wt{G}_{n+1}=(f^*)^{-1}\wt{G}_n$ and $\wt{G}_n\subset \wt{G}_{n+1}$, two distinct puzzle pieces are either disjoint or nested; puzzle pieces of depth $n+1$ are mapped onto those of depth $n$.

Any puzzle piece $P$ in $\mc{P}^*_0$ is a Jordan domain. Indeed, $P$ is a complementary component of the graph $\partial \wt{G}_0$. Since all cut points of $G_0$ are Fatou centers by Proposition \ref{prop:invariant} (3), the graph $\partial \wt{G}_0$ has no cut points as illustrated in Figure \ref{fig:cut}. Therefore, $P$ is a Jordan domain. In general, for puzzle pieces of depth $n$, we have the following fact:
\begin{prop}\cite[Proposition 5.1]{WYZ}\label{prop:Jordan}
	For any $n\geq 1$, every puzzle piece $P$ in $\mc{P}_n$ is a Jordan domain. Furthermore, if $P\subseteq f(P)$, then $\infty\in \partial P$ and $f:P\to f(P)$ is conformal. 	
\end{prop}
This proposition implies the graph $\partial \wt{G}_n$ has no cut points. Hence cut points of $G_n$ can only possible be the centers of components of $B_{f^*}$. Condition (3) of Proposition \ref{prop:invariant} holds for $G_n$. 

Now, we construct puzzle pieces for an arbitrary map $f\in\mc{N}_d$. Recall that, associated to $f$, there is a Newton map $f^*\in \mc{N}_d^*$, a quasiconformal map $h:\olC\to\olC$ and a set $\wt{B}_f\Subset B_f$ consisting of topological disks such that $f$ is conjugate to $f^*$ on $\olC\setminus \wt{B}_f$ via the map $h$. 

Let $G_n, \mc{D}_n, \wt{G}_n$ and $\mc{P}_n^*$ be the sets defined as above for the map $f^*$. Then one can transfer puzzle pieces of $f^*$ to that of $f$ by $h$. Indeed, since $h(\wt{B}_f)\Subset B_{f^*}$ and $B_{f^*}=\bigcup_{n=0}^{+\infty}\mc{D}_n$, there is a large integer $n_0$ such that $h(\wt{B}_f)\Subset \mc{D}_{n_0}$. Let $\mc{P}_n$ be the collection $\{h^{-1}(P^*), P^*\in \mc{P}^*_{n+n_0}\}$ for $n\geq 0$. An element in $\mc{P}_n$ is called a \emph{puzzle piece} of depth $n$ for $f$. Clearly the images $h(P), P\in \mc{P}_n$, contained in $\olC\setminus h(\wt{B}_f)$, are puzzle pieces of $f^*$ with depth $n+n_0$. 

\subsection{Nested properties of puzzle pieces}
We first consider the unbounded puzzle pieces. The point $\infty$ always lies in their boundaries. Since $\partial \wt{G}_n\cap \Delta=\Delta\setminus\tu{int}(\mc{D}_n)$, the number of unbounded puzzle pieces of depth $n$ is a constant independent on $n$. This number is $d_0:=\sum_{i=1}^d(d_{B_i}-1)$. Clearly $d_0\geq 3$ as $d$ is assumed greater than two.
Let
$$\mathcal{P}_n(\infty)=\left\{P_{n,1}(\infty),\cdots,P_{n,d_0}(\infty)\right\}\subset \mc{P}_n$$ be the collection of all unbounded puzzle pieces of depth $n$, numbered in such a way that $P_{n+1,k}(\infty)\subset P_{n, k}(\infty)$ and $P_{n,1}(\infty),\ldots,P_{n, d_0}(\infty)$ are in the anti-clockwise order around the point $\infty$. Then by Proposition \ref{prop:Jordan} there are conformal maps
\begin{equation}\label{eq:conformal}
	f:P_{n+1,k}(\infty)\to P_{n,k}(\infty).
\end{equation}
Let $P_n(\infty)$ be the interior of $\bigcup_{k=1}^{d_0}\ol{P_{n,k}(\infty)}$. Clearly $\partial P_n(\infty)=\bigcup_{k=1}^{d_0}\partial P_{n,k}(\infty)\setminus \Delta$.
\begin{lem}[Nested properties for unbounded puzzles I]\label{lem:nest}
	 The following statements hold.
	\begin{itemize}
	\item[(1)] There exists a large integer $N$ such that $P_{N}(\infty)\Subset {P_0(\infty)}$.
	\item[(2)] $\bigcap P_n(\infty)=\{\infty\}$.
		
	\end{itemize}
\end{lem}
\begin{proof}
	Without loss of generality, we assume $f\in\mc{N}_d^*$. Since $\partial\wt{G}_n\cap \Delta=\Delta\setminus\tu{int}(\mc{D}_n)$, the set $\partial P_{n, k}(\infty)\setminus\Delta$ is an open Jordan arc, denoted by $\Gamma_{n, k}$, for any $n\geq 0$ and $1\leq k\leq d_0$. It holds that $f(\Gamma_{n+1,k})=\Gamma_{n,k}$ and $\Gamma_{n+1,k}\subseteq \ol{P_{n,k}(\infty)}$.
		
		 Since $\Gamma_{n,k}\subset \partial \wt{G}_n$ consists of finitely many segments from perperiodic internal rays, for some large integer $n_k$ we have $\Gamma_{n_k, k}\cap\Gamma_{0, k}=\es$ and so $\Gamma_{n_k, k}\subset P_{0,k}(\infty)$. Let $N:=\tu{max}\{n_1,\ldots,n_{d_0}\}$. Note that $\partial P_{n}(\infty)=\bigcup_{1\leq k\leq d_0} \Gamma_{n,k}$ for all $n$. Thus $\partial P_{N}(\infty)\subset P_{0}(\infty)$. The statement of (1) holds.
		 
		 We claim that the set $P_{N}(\infty)$ is a Jordan domain, which forms a neighborhood of $\infty$. Indeed, since the $d_0$ open arcs $\Gamma_{N, k}$ lie in the interiors of distinct puzzle pieces among $\mc{P}_N(\infty)$, they are pairwise disjoint. So their union forms a Jordan curve, which is the boundary of $P_N(\infty)$. It also follows that $P_n(\infty), n>N,$ are Jordan domains, as they are pullbacks of $P_N(\infty)$.
		 
		 For the statement of (2), consider the conformal map $g:=f^N: P_N(\infty)\to P_0(\infty)$. The intersection $\bigcap P_n(\infty)$ is enclosed by pairwise disjoint annuli $H_i:=g^{-i}(P_0(\infty)\setminus\ol{P_N(\infty)})$. These annuli have the same modulus $\tu{mod}(H_0)$. The Gr\"{o}tzsch inequality implies that $$\tu{mod}\left(P_0(\infty)\setminus\bigcap\ol{P_n(\infty)}\right)=\infty.$$
		Thus $\bigcap\ol{P_n(\infty)}=\bigcap P_n(\infty)$ is a singleton. The proof is complete.		 
\end{proof}

The \emph{grand orbit} of $\infty$ is denoted by $\Omega_f:=\bigcup_{k\geq 1}f^{-k}(\infty)$. For each $z\in \Omega_f$ with $f^{n_0}(z)=\infty$ for some minimal integer $n_0\geq 1$, let $P_{n+n_0}(z)$ be the component of $f^{-n_0}(P_n(\infty))$ containing the point $z$ for $n\geq 0$. The collection of puzzle pieces of depth $n+n_0$ covered by $P_{n+n_0}(z)$ is denoted by $\mc{P}_{n+n_0}(z)$. Clearly $\#\mc{P}_{n+n_0}(z)=\tu{deg}_z(f^{n_0})\cdot d_0$.
 
We remark that the nested properties in Lemma \ref{lem:nest} also hold for $P_{n}(z), n\geq {n_0}, z\in \Omega_f$. One should be careful that by definition $P_n(z)$ is not a puzzle piece but contains at least three puzzle pieces in this case.

For each $z\in\mc{J}_f\setminus \Omega_f$, let $P_n(z)$ be the puzzle piece of depth $n$ containing the point $z$. Obviously we have $f(P_{n+1}(z))=P_n(f(z))$.

\begin{lem}[Nested properties for unbounded puzzles II]\label{lem:nest-infty}
	Let $N$ be an integer with $P_{N}(\infty)\Subset P_{0}(\infty)$. Let $z\in \mc{J}_f\setminus\Omega_f$ be a point in $P_{N, m}(\infty)\in\mc{P}_{N}(\infty)$. Then there is an integer $s\geq 0$   with the following property:
	$$P_{N+1}(f^s(z))\Subset P_{0, m}(\infty)\in\mc{P}_{0}(\infty)$$
	and $f^s: (P_{N+s+1}(z), P_{s, m}(\infty))\to (P_{N+1}(f^s(z)), P_{0,m}(\infty)$ is conformal.
\end{lem}
\begin{proof}
	From the construction of puzzle pieces around $\infty$, we know that 
	$$P_{N,m}(\infty)=\bigcup_{k\geq 0}P_{N+k,m}(\infty)\setminus P_{N+k+1,m}(\infty).$$
	Let $s$ be the minimal integer such that $z\in P_{N+s,m}(\infty)\setminus P_{N+s+1,m}(\infty)$. Then $$P_{N+s+1}(z)\subset P_{N+s,m}(\infty)\setminus P_{N+s+1,m}(\infty).$$ By \eqref{eq:conformal}, we obtain the required conformal map $f^s$. Since $$P_{N+1}(f^s(z))\subset P_{N,m}(\infty)\sm P_{N+1,m}(\infty),$$ the condition $P_N(\infty)\Subset {P_0(\infty)}$ implies that $P_{N+1}(f^s(z))\Subset P_{0, m}(\infty).$ The proof is complete.
\end{proof}

\begin{lem}[Nested properties for puzzles around $z\in \mc{J}_f\setminus\Omega_f$]\label{lem:nest-finite}
	Let $z$ be a point in $\mc{J}_f\setminus \Omega_f$ and $P_k(z)$ be a puzzle piece of depth $k\geq 0$. There exists a large integer $n_k$ such that $P_{n_k}(z)\Subset P_k(z)$.
\end{lem}

\begin{proof}
	We first claim that for two puzzle pieces $P$ and $Q$, $\partial P\cap\partial Q=\es$ if and only if the two sets $\partial P\cap\mc{J}_f$ and $\partial Q\cap \mc{J}_f$ are disjoint. To show $\partial P\cap\partial Q=\es$, note that a point in $\partial P\cap\partial Q\setminus \mc{J}_f$ must be contained in a unique internal ray $R$. The landing point of $R$ belongs to both $\partial P$ and $\partial Q$. Thus $\partial P$ and $\partial Q$ have at least a common point in $\mc{J}_f$, which is a contradiction.
	
	Assume the conclusion of this lemma is not true. Note that $P_{n}(z)\subset P_{m}(z)$ whenever $n>m$ and $P_{n}(z)\cap\mc{J}_f$ consists of only finitely many points in $\Omega_f$.  By the above claim, there is a point $\xi\in \partial P_{n}(z)\cap \Omega_f$ for all $n\geq k$. Without loss of generality, one may assume $\xi=\infty$. By Lemma \ref{lem:nest}\,(2), the set $\bigcap\ol{P_n(z)}$ that contains both $\infty$ and $z$ is a singleton. Hence $z=\infty$, a contradiction. The proof is complete.
\end{proof}


\begin{prop}\label{prop:shrinking}
	Let $f\in\mc{N}_d$ be a non-renormalizable Newton map. Then for any preperiodic point $z\in \mc{J}_f\setminus\Omega_f$, $\bigcap P_n(z)=\{z\}$.
\end{prop}
\begin{proof}
	It suffices to consider the case that $z$ is periodic and $z\neq \infty$. Let $e(z):=\bigcap P_n(z)=\bigcap\ol{P_n(z)}$ and $p>1$ be the period of $z$. Let $n_0$ be a large integer such that the annulus $P_{n_0}(z)\setminus e(z)$ is disjoint from critical points of $f^p$. By Lemma \ref{lem:nest-finite}, there exists a puzzle piece $P_{n_0+kp}(z)$ compactly nested in $P_{n_0}(z)$ for some large integer $k$. Suppose $e(z)$ is not a singleton. Then $g=f^{k}: P_{n_0+kp}(z)\to P_{n_0}(z)$ is a polynomial-like mapping with connected filled-in Julia set $K(g)=e(z)$. This is a contradiction since $f$ is non-renormalizable.
\end{proof}

\section{On wandering points satisfying the elevator condition}\label{sec:wandering}
A point $z$ in $\mc{J}_f$ is called \emph{wandering} if its orbit $\tu{orb}(z):=\bigcup_{n\geq 0} f^n(z)$ is infinite. In this section we aim to show that wandering points satisfy the \emph{elevator condition} except the ones in $X_f$; see Corollary \ref{cor:mes}.

Throughout this section we assume $f$ is a non-renormalizable Newton map.

\begin{defi}[The bounded degree condition and elevator condition] A wandering point $z\in\mc{J}_f$ is said to satisfy the \emph{bounded degree condition}, if there exist puzzle pieces $V_k$ and $V$ with $z\in V_k$ and integers $n_k\to \infty$ as $k\to \infty$ such that 
$$\tu{deg}(f^{n_k}:V_k\to V)\leq \delta\tu{ for all }k\geq 1.$$ Furthermore, if there exists a puzzle piece $U$ compactly nested in $V$ such that $U$ contains infinitely many points of $f^{n_k}(z)$, then the point $z$ is said to satisfy the \emph{elevator condition}. In this case, by passing to a subsequence, one may assume $f^{n_k}(z)\in U$, and the triple $(U, V, n_k)$ is called an \emph{elevator} for the point $z$. 
\end{defi}
We remark that the elevator condition can be generalized in two aspects: (1) the topological disks $U$ and $V$ are not necessary to be puzzle pieces; (2) the point $z$ is allowed to be (pre)-periodic point and satisfies that the index set $\{n_k: f^{n_k}(z)\in U\}$ is infinite. In sense of this definition, (pre-)periodic points in $\mc{J}_f$ including those in $\Omega_f$ satisfy the elevator condition.

\begin{lem}\label{lem:elevator}A wandering point $z\in\mc{J}_f$ with the bounded degree condition satisfies the elevator condition.
\end{lem}
\begin{proof}
	By definition, we may assume the puzzle piece $V$ contains infinitely many points $f^{n_k}(z)$. If there is an accumulation point $x$ of $\{f^{n_k}(z)\}$ such that $x\in \mc{J}_f\setminus\Omega_f$, then by Lemma \ref{lem:nest-finite}, we obtain a puzzle piece $U$ containing $x$ that is compactly nested in $V$. Since $x$ is an interior point of $U$, the set $\{f^{n_k}(z):f^{n_k}(z)\in V\}$ is infinite. 
	
	We are reduced to the case that the $\omega$-limit set of $\{f^{n_k}(z)\}$ belongs to $\Omega_f$. By considering $\{f^{n_k+l}(z)\}$ if necessary, one may assume $\infty$ is an accumulation point of $\{f^{n_k}(z)\}$ and $V=P_{0,m}(\infty)\in\mc{P}_0(\infty)$. In this case, by Lemma \ref{lem:nest} we assume further $P_N(\infty)\Subset P_0(\infty)$ for some integer $N$. Note that $P_{N, m}(\infty)$ contains infinitely many points of $\{f^{n_k}(z)\}$ as well. By passing to a subsequence, assume $f^{n_k}(z)\in P_{N, m}(\infty)$. According to Lemma \ref{lem:nest-infty} for each point $z_{n_k}:=f^{n_k}(z)$, we obtain a puzzle piece $P_{N+1}(z_{n_k+s_k})\Subset P_{0,m}(\infty)$ and a conformal map:
	$$g_k:=f^{s_k}: (P_{N+s_k+1}(z_{n_k}), P_{s_k, m}(\infty))\to (P_{N+1}(z_{n_k+s_k}), P_{0,m}(\infty)).$$
	Note that $P_{s_k, m}(\infty)\subset P_{0,m}(\infty)$.

	Let $V_{k}$ and $\wt{V}_{k}$ be the components of $f^{-n_k}(P_{s_k, m}(\infty))$ and $f^{-n_k}(P_{0,m}(\infty))$ containing the point $z$, respectively. Consider the two maps $$h_k:=f^{n_k}: (V_k,\wt{V}_{k})\to (P_{s_k,m}(\infty), P_{0,m}(\infty))$$ and $g_k: P_{s_k, m}(\infty)\to P_{0,m}(\infty)$. The composition 	
	$g_k\circ h_k=f^{n_k+s_k}: V_{k}\to P_{0,m}(\infty)$ has a uniformly bounded degree. Note that $P_{N+1}(z_{n_k+s_k})\Subset P_{0,m}(\infty)$ and the number of puzzle pieces of depth $N+1$ is finite. Thus one of those $P_{N+1}(z_{n_k+s_k})$, say $U$, contains infinitely many points of $\{z_{n_k+s_k}\}$. Hence by passing to a subsequence, $(U, V, n_k+s_k)$ is an elevator for the point $z$. The proof is complete.	 
\end{proof}

Let $z\in\mc{J}_f$ be a wandering point and $P$ be a puzzle piece. The \emph{first entry time} of $z$ into $P$, denoted by $r_z(P)$, is the minimal integer $k\geq 1$ such that $f^k(z)\in P$. If no such integer exists, we set $r_z(P)=\infty$. If $r_z(P)\neq \infty$, we de{\normalsize }note by $L_z(P)$ the unique puzzle piece containing $z$ such that $f^{r_z(P)}(L_z(P))=P$. 

Recall that $\kappa$ is the number of critical points in $\mc{J}_f$.
\begin{lem}[Properties of the first entrance for points]\label{lem:first-time}
	Let $z\in\mc{J}_f$ be a wandering point. Let $P$ be a puzzle piece such that the first entry time $r=r_z(P)$ is finite. Then the degree of $f^r: L_z(P)\to P$ is bounded by $d^\kappa$.
\end{lem}
\begin{proof}
	We write $Q_k:=f^k(L_z(P)), 0\leq k\leq r$. Claim that the $r$ puzzle pieces $Q_0,\cdots Q_{r-1}$ are pairwise disjoint. For otherwise, assume $Q_{r_1}\cap Q_{r_2}\neq \emptyset$ with $r_1<r_2$. Then $f^{r_1}(z)\in Q_{r_1}\subset Q_{r_2}$. We pull back $Q_{r_2}$ along the orbit $z\mapsto f(z)\mapsto\cdots\mapsto f^{r_1}(z)$. Then we obtain a puzzle piece $W$ containing $z$ whose depth is strictly less than that of $Q_0$ as $r_1< r_2$. Hence $f^{r_1+r-r_2}(W)=P$ and $z\in W$ with $r_1+r-r_2<r_z(P)$. It is a contradiction. 
	
	By the claim, each critical point of $f$ appears at most once in the orbit $Q_0,\cdots, Q_{r-1}$. Therefore, the degree of $f^r:L_z(P)\to P$ is bounded by $d^\kappa$. 
\end{proof}
\begin{defi}Let $z\in\mc{J}_f$ be a wandering point. A point $x\in \mc{J}_f$ (possibly $x\in\Omega_f$) is called a \emph{combinatorial accumulation point} of $\tu{orb}(z)$, written $x\in\omega_{\tu{comb}}(z)$, provided that, for any $n\geq 0$, the annulus $P_n(x)\setminus\bigcap_k\ol{P_k(x)}$ contains infinitely many points of $\tu{orb}(z)$.
\end{defi}
Recall that the $\omega$-limit set $\omega(z)\subset \olC$ is the set of accumulation points of $\tu{orb}(z)$ in sense of the spherical metric. The notion of the combinatorial accumulation set $\omega_{\tu{comb}}(z)$ coincides with the notion of the standard $\omega$-limit set $\omega(z)$ but for the topology generated by the set of all puzzle pieces together with those $P_n(x), x\in \Omega_f$.
Lemma \ref{lem:comb} describes some basic properties of combinatorial accumulation sets. 
\begin{lem}[Properties of $\omega_{\tu{comb}}(z)$]\label{lem:comb}
	Let $z\in\mc{J}_f$ be a wandering point. The following statements hold:
	\begin{enumerate}
		\item for any $x\in\mc{J}_f$, $x\in\omega_{\tu{comb}}(z)\Leftrightarrow$  for any $n\geq 0$, $P_n(x)$ contains infinitely many points of $\tu{orb}(z)$;
		\item for any $x\in\mc{J}_f$,
		$x\notin \omega_{\tu{comb}}(z)\Leftrightarrow P_{n_0}(x)\cap\tu{orb}(f^{n_0}(z))=\es\tu{ for some } n_0\geq 0.$
		\item $f(\omega_{\tu{comb}}(z))\subseteq \omega_{\tu{comb}}(z)=\omega_{\tu{comb}}(f(z))$;
		\item $\omega(z)\subseteq \omega_{\tu{comb}}(z)$;
		\item $\omega_{\tu{comb}}(x)\subseteq\omega_{\tu{comb}}(z)$ for any wandering point $x\in\omega_{\tu{comb}}(z)$;
		\item $\omega_{\tu{comb}}(z)$ contains infinitely many points.
	\end{enumerate}
\end{lem}
\begin{proof}
	(1) To show $x\in\omega_{\tu{comb}}(z)$, it suffices to prove that $e(x):=\bigcap P_k(x)$ contains at most finitely many points of $\tu{orb}(z)$. Actually, $\#e(x)\cap\tu{orb}(z)\leq 1$. For otherwise, assume $f^{n_0}(z), f^{n_1}(z)\in e(x)$ with $n_1-n_0=p\geq 1$. Since $f(e(x))=e(f(x))$, we have $$e(x)=e(f^{n_1}(z))=f^{p}(e(f^{n_0}(z)))=f^{p}(e(x)).$$
	By Lemmas \ref{lem:nest}\,(2) and \ref{lem:nest-finite}, one may take two puzzle pieces $U\Subset V$ containing $e(x)$ with the difference of their depths equal to $kp$ for some large integer $k$. Assume further that $V\setminus e(x)$ is disjoint from $\tu{crit}(f)$. Since $f$ is non-renormalizable, the map $g=f^{kp}: U\to V$ is conformal. Note that $e(x)$ is enclosed by disjoint annuli $g^{-i}(V\setminus\ol{U})$, which are of the same moduli $\tu{mod}(V\setminus\ol{U})$. By the Gr\"{o}tzsch inequality, $\tu{mod}(V\setminus e(x))=\infty$. Hence $e(x)$ is a singleton and so $x=f^{n_1}(z)=f^{n_2}(z)$. This implies the point $z$ is preperiodic, a contradiction.
	
	It is not hard to see that statements (2)--(5) are consequences of statement (1). We left the proofs to the reader. 
	
	To see (6), we argue by contradiction and assume $\omega_{\tu{comb}}(z)$ is finite. Since the set $\omega_{\tu{comb}}(z)$ is $f$-invariant, without loss of generality, we may assume $\omega_{\tu{comb}}(z)=\tu{orb}(x)$ for some preperiodic point $x\in\mc{J}_f$. Suppose $f^{m+p}(x)=f^{m}(x)$ for some minimal integers $m\geq 0$ and $p\geq 1$. By Lemma \ref{lem:nest}\,(2) and Proposition \ref{prop:shrinking}, there is a large integer $n_0$ such that the disks $P_{n_0-p}(x_i)$ around $x_i:=f^i(x)$ are pairwise disjoint. Since $\omega_{\tu{comb}}(z)=\tu{orb}(x)$, by definition the orbit $\tu{orb}(f^{k_0}(z))$ for some large integer $k_0$ is covered by the disks $P_{n_0}(x_i)$.
	
	Claim that $f^{p}(\xi)\in P_{n_0}(x_m)$ for any $\xi\in P_{n_0}(x_{m})\cap\tu{orb}(f_{k_0}(z))$. Indeed, the point $f^p(\xi)$ belongs to $P_{n_0-p}(x_m)$, which is disjoint from the other disks $P_{n_0}(x_i)$, $i\neq m$. Since $f^p(\xi)\in\tu{orb}(f^{k_0}(z))$, the choice of $n_0$ leads to $f^p(\xi)\in P_{n_0}(x_m)$.
	
	By the claim, it follows that $\xi\in e(x_m)$. Since $e(x_m)=\{x_m\}$ and $\xi\in\tu{orb}(z)$, the point $z$ is preperiodic. This is a contradiction. The proof is complete.
\end{proof}

\begin{lem}\label{lem:preperiodic}
	Let $z\in\mc{J}_f$ be a wandering point. If $\omega_\tu{comb}(z)$ contains a preperiodic point $x$, then the point $z$ satisfies the bounded degree condition.
\end{lem}
\begin{proof}
	Since $f(x)\in \omega_{\tu{comb}}(z)$, for simplicity we assume $x$ is periodic. As $\omega_{\tu{comb}}(z)$ contains infinitely many points, take a point $y\in \omega_{\tu{comb}}(z)\setminus\tu{orb}(x)$. Note that it is possible that $y\in\Omega_f$. By Proposition \ref{prop:shrinking}, there exists a puzzle piece $U:=P_{n_0}(y)$ for a large integer $n_0$ such that 
	\begin{equation}\label{eq:U}
			U\cap \tu{orb}(x)=\es\tu{ and }\#U\cap\tu{orb}(z)=\infty,
	\end{equation}
if $y\in\mc{J}_f\setminus\Omega_f$. If $y\in \Omega_f$, by Lemma \ref{lem:nest} such a puzzle piece $U$ satisfying \eqref{eq:U} can be chosen as $P_{n_0,i}(y)\in\mc{P}_{n_0}(y)$. We then consider two cases: $x\neq \infty$ and $x=\infty$. 

In the case that $x\neq \infty$, let $r_n$ be the first entry time of $z$ into $P_n(x), n\geq n_0$. The point $f^{r_n}(z)$ eventually falls into the puzzle piece $U$ by \eqref{eq:U}. We denote by $s_n$ the first entry time of $f^{r_n}(z)$ into $U$. Since $U\cap \tu{orb}(x)=\es$. The puzzle piece $Q_n:=L_{f^{r_n}(z)}(U)$ contains $f^{r_n}(z)$ but avoids the point $x$. Thus $Q_n\subseteq P_n(x)$.  Note that by Lemma \ref{lem:first-time}, the degrees of maps
$$g_n:=f^{r_n}:P_{n+r_n}(z)\to P_n(x)\tu{ and }h_n:=f^{s_n}: Q_n\to U$$
are bounded by $d^\kappa$. Let $U_n$ be the component of $g_n^{-1}(Q_n)$ containing the point $z$. Then the composition
$$h_n\circ g_n=f^{r_n+s_n}:U_n\to U$$
has bounded degree $d^{2\kappa}$. It remains to show that $r_n\to\infty$ as $n\to\infty$. It is equivalent to prove that $f^{r_n}(z)$ is disjoint from $e(x):=\bigcap P_k(x)$ for all $n$. For otherwise, $f^{s_n}(e(x))=e(y)\subset U$ and so $f^{s_n}(x)\in U$, contradicting \eqref{eq:U}.

In the case that $x=\infty$, we choose a sequence of unbounded puzzle pieces $P_n:=P_{n, i_n}(\infty)\in\mc{P}_n(\infty)$, such that $\#P_n\cap \tu{orb}(z)=\infty$. As above, let $r_n$ be the first entry time of $z$ into $P_n$ and $s_n$ the first entry time of $f^{r_n}(z)$ into $U$. Then we obtain puzzle pieces $Q_n, U_n$ and two maps $g_n, h_n$ exactly as above such that the degree of the composition $h_n\circ g_n=f^{r_n+s_n}: U_n\to U$ is bounded by $d^{2\kappa}$ for each $n\geq n_0$. By similar arguments we have $r_n\to\infty$ as $n\to\infty$. This implies the point $z$ satisfies the bounded degree condition. The proof is complete.	
\end{proof}

\begin{lem}\label{lem:nonrecurrent}
	Let $z\in\mc{J}_f$ be a wandering point. If there exists a wandering point $x\in\omega_{\tu{comb}}(z)$ such that $\omega_\tu{comb}(x)\neq \omega_{\tu{comb}}(z)$, then the point $z$ satisfies the bounded degree condition.
\end{lem}
\begin{proof}
	By Lemma \ref{lem:preperiodic} we may assume all points in $\omega_{\tu{comb}}(z)$ are wandering. By condition, $\omega_{\tu{comb}}(x)\subset \omega_{\tu{comb}}(z)$. We take a wandering point $y\in\omega_{\tu{comb}}(z)\setminus\omega_{\tu{comb}}(x)$. Then by Lemma \ref{lem:comb}\,(2) there exists a puzzle piece $U:=P_{n_0}(y)$ such that	
	$U\cap \tu{orb}(f^{n_0}(x))=\es$. Since $\omega_{\tu{comb}}(x)=\omega_{\tu{comb}}(f(x))$ and $f(x)\in\omega_{\tu{comb}}(z)$, for simplicity we may replace the point $x$ by $f^{n_0}(x)$ and assume $U$ satisfies \eqref{eq:U}.
	
	Now the proof goes similarly as Lemma \ref{lem:preperiodic}. Let $r_n$ be the first entry time of the point $z$ into $P_n(x), n\geq n_0$. Let $s_n$ be the first entry time of $f^{r_n}(z)$ into the puzzle $U$. Then we obtain a map $h_n\circ g_n=f^{r_n+s_n}$ from a puzzle piece $U_n$ (containing the point $z$) to the fixed puzzle $U$ with its degree less than $d^{2\kappa}$ for each $n\geq n_0$. Moreover, $r_n\to\infty$ as $n\to\infty$. This implies that $z$ satisfies the bounded degree condition. 
\end{proof}

We denote by $\omega\tu{Crit}(z):=\tu{crit}(f)\cap \omega_{\tu{comb}}(z)$, i.e., the set of all critical points in $\omega_{\tu{comb}}(z)$. 
\begin{lem}\label{lem:wcrit}
	Let $z\in\mc{J}_f$ be a wandering point. If $\omega\tu{Crit}(z)=\emptyset$, then 
 the point $z$ satisfies the bounded degree condition. 
\end{lem}
\begin{proof}
 Without loss of generality, we may assume $\tu{orb}(z)\cap \tu{crit}(f)=\es$. By condition, there exists an integer $n_0$ such that $P_{n_0}(c)\cap \tu{orb}(z)=\emptyset$ for all $c\in\tu{crit}(f)$. The nested property of puzzle pieces implies that
	$$P_{n+n_0}(z), P_{n+n_0-1}(f(z)), \cdots, P_{n_0}(f^{n}(z)), n\geq 1,$$ are disjoint from $P_{n_0}(c), c\in \tu{crit}(f)$. Let $U$ be a puzzle piece of depth $n_0$ such that $f^{n_k}(z)\in U$ with $n_k\to\infty$ as $k\to \infty$.
	Then the maps: $f^{n_k}:P_{n_0+n_k}(z)\to U$ are conformal.
 This implies that $z$ satisfies the bounded degree condition.
\end{proof}

Let $c\in\mc{J}_f$ be a wandering critical point. The puzzle piece $P_{n+k}(c)$ is called a \emph{successor} of $P_n(c)$ if $f^k(P_{n+k}(c))=P_n(c)$ and the itinerary
$$P_{n+k}(c), f(P_{n+k}(c)), \cdots, f^k(P_{n+k}(c))=P_{n}(c)$$
	meets each critical point at most twice. Clearly the degree $f^k:P_{n+k}(c)\to P_n(c)$ is bounded by $d^{2\kappa-1}$.

\begin{lem}\label{lem:successors}
	Let $z\in\mc{J}_f$ be a wandering point. If there exists a critical point $c\in\omega\tu{Crit}(z)$ and an integer $n_0$ such that $P_{n_0}(c)$ has infinitely many successors, then the point $z$ satisfies the bounded degree condition.
\end{lem}
\begin{proof}
	If $P_{n_0}(c)$ has infinitely many successors $P_{n_0+k_n}(c)$ with $k_n\to\infty$ as $n\to\infty$, then the degrees of maps
	$$f^{k_n}:P_{n_0+k_n}(c)\to P_{n_0}(c)$$
	are bounded by $d^{2\kappa-1}$ by definition. Let $r_n$ be the first entry time of $z$ into $P_{n_0+k_n}(c)$ and $V_n$ be the component of $f^{-r_n}(P_{n_0+k_n})$ containing the point $z$. Then the composition $f^{k_n}\circ f^{r_n}: V_n\to P_{n_0}(c)$ has degree less than $d^{3\kappa-1}$. The proof is complete.
\end{proof}

By Lemmas \ref{lem:preperiodic}, \ref{lem:nonrecurrent}, \ref{lem:wcrit} and \ref{lem:successors}, we can restrict ourself on the set $X_f$, which consists of wandering points $z\in\mc{J}_f$ satisfying
\begin{itemize}
	\item[(1)] $\omega\tu{Crit}(z)\neq\es$;
	\item[(2)] any point $x$ in $\omega_{\tu{comb}}(z)$ are wandering and $\omega_{\tu{comb}}(x)=\omega_{\tu{comb}}(z)$;
	\item[(3)] $P_n(c)$ has only finitely many successors for any $c\in\omega\tu{Crit}(z)$ and $n\geq 0$.
\end{itemize}
The critical points in the above set $\omega\tu{Crit}(z)$ are called \emph{persistently recurrent}. The following corollary is an immediate consequence of Lemmas \ref{lem:elevator},  \ref{lem:preperiodic}, \ref{lem:nonrecurrent}, \ref{lem:wcrit} and \ref{lem:successors}.
\begin{cor}\label{cor:mes}
	Let $f$ be a non-renormalizable Newton map. Then every wandering point in $\mc{J}_f\setminus X_f$ satisfies the elevator condition.
\end{cor}

\section{On persistently recurrent critical points}\label{sec:recurrent}
We do not know whether a persistently recurrent critical point satisfies the bounded degree condition. However, in this section we shall construct a sequence of nested puzzle pieces $(K_n^-, K_n, K_n^+)$, called \emph{enhanced nest}, around the point $c$ such that these puzzle pieces satisfy certain bounded degree condition and their associated annuli have uniformly positive lower moduli; see Corollary \ref{cor:nest} and Theorem \ref{thm:modulus}. This construction relies on well developed theories on peristently recurrent critical points; see \cite{KSS,KS,TLP,QY}.

Throughout this section, we assume $f$ is a non-renormalizable Newton map, and fix a point $z\in X_f$ and a critical point $c\in \omega\tu{Crit}(z)$. Let us denote by $\mathcal{PC}(z):=\ol{\bigcup_{c\in \omega\tu{Crit}(z)}\tu{orb}(c)}$. 

In \cite{KSS} the authors introduced two operators, written as $\mc{A}$ and $\mc{B}$, acting on critical puzzle pieces and producing new critical puzzle pieces. These puzzle pieces have ``bounded degree" and their associated annuli avoid the set $\mathcal{PC}(z)$. Precisely, 

\begin{proposition}\label{p:pteAB} Let $I$ be a puzzle piece containing the point $c$. Then $\mathcal{A}(I)$ and $\mathcal{B}(I)$ are  puzzle pieces with the following properties\,:
	
	\begin{enumerate}
		\item [(1)] $c\in\AA(I)\subset \BB(I)\subset I$ and $\BB(I)\setminus \ol{\AA(I)}$ is disjoint from the set $\mathcal{PC}(z)$\,;
		\item [(2)] $f^{a}(\AA(I))=I$ and $f^{b}(\BB(I))=I$ for some integers $a=a(I)$ and $b=b(I)$;
		\item [(3)] the degree of  $  f^{a}:\AA(I)\to I $ is less than $d^{\kappa^2+\kappa}$.
		\item[(4)] the degree of $ f^{b}:\BB(I)\to I $ is less than $d^{\kappa^2}$;
	\end{enumerate}
\end{proposition}
Recall that $\kappa$ is the number of critical points in $\mc{J}_f$. We refer to~\cite[pp.802-805]{KSS} or \cite[pp.52-54]{QY} or \cite[Lemma 4.4]{TLP} for the proof of Proposition \ref{p:pteAB}.

By condition, each puzzle piece $I$ containing the point $c$ has at most finitely many successors. Let $\mc{D}(I)$ be the last successor of $I$. Recall that $L_c(I)$ is the unique puzzle piece containing $c$ such that $f^{r}(L_c(I))=I$ with $r:=r_c(I)$ the first entry time of $c$ into $I$. Clearly $L_c(I)$ is a successor of $I$.
An important fact is that
\begin{lem}\cite[Appendix A]{TLP}\label{lem:two}
	Every puzzle piece containing the point $c$ has at least two successors.
\end{lem}

The puzzle pieces $K_n^-$ and $K_n$ around $c$ are defined inductively by $K_n^-:=\mathcal{A}\mathcal{A}\mathcal{D}^{\tau}(K_{n-1})$ and $K_n:=\mathcal{B}\mathcal{A}\mathcal{D}^{\tau}(K_{n-1})$, starting from  $K_0:=Q$. Here $Q$ is a puzzle piece compactly nested in $P_0(c)$ by Lemma \ref{lem:nest-finite}, and $\tau$ is a large integer. Actually taking $\tau=\kappa+2$ is enough.

Assume $f^{q_n}:K_n\to \mathcal{A}\mathcal{D}^{\tau}(K_{n-1})$ for some integer $q_n>0$. Then $K_n^+$ is defined as the component of $f^{-q_n}(\mathcal{B}\mathcal{D}^{\tau}(K_{n-1}))$ containing the point $c$.

By definition, the corollary follows immediately.
\begin{cor}\label{cor:nest}
	
	The puzzle pieces $(K_n^-,K_n,K'_n)$ have the following properties\,:
	
	\begin{enumerate}\item [(1)] $c\in K_n^-\subset K_n\subset K^+_n\subset K_{n-1}^-$;
		\item[(2)] $f^{p_n^-}(K_n^-)=K_{n-1}$,  $f^{p_n}(K_n)=K_{n-1}$, $f^{p^+_n}(K^+_n)=K_{n-1}$ for some $p_n^-, p_n$ and $p_n^+$;
		\item[(3)] $\tu{deg}(f^{p_n}|_{K_n})\le \delta=\delta(d,\kappa)$;
		\item[(4)] $K_n^+\setminus\ol{K_n}$  and $K_n\setminus\ol{K_n^-}$ are  annuli  (possibly degenerated) disjoint from the set $\mathcal{PC}(z)$.
	\end{enumerate}
\end{cor}

Let $I$ be a puzzle piece. The \emph{first return time} of $I$ is the minimal integer $r\geq 1$ such that $f^r(I)\cap I\neq \es$, or equivalently, $r$ is the smallest integer such that $I\cap f^{-r}(I)\neq \es$. We denote this $r$ by $r(I)$. This notation is, in spirit, same as the first entry time for points. Since the iterated image $f^n(I)$ will eventually cover the sphere $\olC$, the number $r(I)$ exists. 

\begin{lem}[Properties of the first return for puzzles]\label{lem:property}
	Let $I'\subset I$ be two puzzle pieces. Then
	\begin{enumerate}
		\item $r(I)\leq r(I')$;
		\item If $f^s(I')=I$ for some $s>0$, then $s\geq r(I')$ and the equality holds when $I'$ is a successor of $I$.
		\item $r(\mc{D}(I))\geq 2 r(I)$.
	\end{enumerate}
\end{lem}
\begin{proof}
	Statements (1) and (2) hold by definition. For (3), assume $I'=\mc{D}(I)$ and $f^s(I')=I$. Let $s_0$ be the first entry time of $c$ into $I$. By Lemma \ref{lem:two}, $s>s_0$ and $c\in I'\subset L_c(I)$. Then $f^{s_0}(I')\subset f^{s_0}(L_c(I))=I$. By statements (1) and (2), we have
	$$r(I')=s_0+(s-s_0)\geq r(L_c(I))+r(f^{s_0}(I'))\geq 2r(I).$$ 
	The proof is complete.
\end{proof}
We denote by $h(I', I)$ the difference of the depths of puzzle pieces $I'$ and $I$.
\begin{cor}\label{cor:infinity}
	$\lim\limits_{n\to \infty}h(K_n^+, K_n)=\lim\limits_{n\to \infty}h(K_n, K^-_n)=\infty$.
\end{cor}
\begin{proof}
	Note that $h(K^+_n,K_n)=p_n-p^+_n=h(K_{n-1},f^{p_n^+}(K_n))\geq r(f^{p_n^+}(K_n))\geq r(K_{n-1})$ by Corollary \ref{cor:nest} and Lemma \ref{lem:property}. Similarly, $h(K_n,K_n^-)\geq r(K_{n-1})$. On the other hand, $r(K_n)\geq r(\mc{D}^\tau (K_{n-1})\geq 2^\tau r(K_{n-1})\geq (2^{\tau })^nr(K_0)$. Thus the lemma follows.	
\end{proof}

\begin{lem}\label{l:Knondeg} The annuli $K^+_n\setminus\ol{K_n}$ and $K_n\setminus\ol{K_n^-}$ are non-degenerated for all large enough $n$.
\end{lem}
\begin{proof}
	Assume $f^{t_n}(K_n)=K_0$. By the choice of $K_0$, we have $K_0\Subset P_0(c).$ Let $R_n$ be the component of $f^{-t_n}(P_0(c))$ containing $c$. Then $K_n\Subset R_n$. Note that $h(K_n, R_n)$ is a constant $h(K_0, P_0(c))$. By Corollary \ref{cor:infinity}, for a large integer $n$, we have $h(K_n, R_n)<h(K_n, K_n^+)$. Since $R_n$ and $K_n^+$ contains a common point $c$, it holds that $R_n\subset K_n^+$. Thus $K_n\Subset K^+_n$. A similar argument implies that $K_n^-\Subset K_n$ for large integers $n$. The proof is complete.
\end{proof}

The Covering Lemma, due to Kahn and Lyubich, is a very powerful tool\,:
in a degenerated situation on moduli, it tells us the relation of certain (sub-)annuli and the pre-images of these annuli under a holomorphic proper map  $g$, when one has some control on the degrees of $g$ over some disks.
We state the lemma now and refer the reader to~\cite{KL1}  for the proof.

\begin{lem}[The Kahn-Lyubich Covering Lemma]\label{lem:KL}
	Let $g:(A, A', U)\to (B, B',V)$ be a degree $\delta$ proper holomorphic map with the following properties:
	\begin{itemize}
		\item[(1)]
		$A\Subset A'\Subset U$ and $B\Subset B'\Subset V$
		are all topological disks;
		\item[(2)]  $g: A\to B$ and $g: A'\to B'$ are proper with degrees less than $D(\leq \delta)$;
		\item[(3)]  $\textup{mod}(B'\setminus \ol{B})\ge \eta \ \textup{mod}(U\setminus \ol{A})$.
	\end{itemize}
	Then there exists a number
	$\epsilon=\epsilon(\eta,\delta)>0$ such that if $\textup{mod}(U\setminus \ol{A})<\epsilon$, we have
	\begin{equation}\label{eq:kL}
	\textup{mod}(U\setminus\ol{A})\geq \frac \eta{2D^2}\  \textup{mod}(V\setminus\ol{B}).
	\end{equation}
\end{lem}

\begin{thm}[Complex bounds]\label{thm:modulus}
	There exists a constant $m>0$ such that $$\tu{mod}(K_n^+\setminus\ol{K_n})\geq m\tu{ and }\tu{mod}(K_n\setminus\ol{K_n^-})\geq m.$$
	for sufficiently large integer $n$.
\end{thm}
The proof of Theorem \ref{thm:modulus} makes essential use of Lemma \ref{lem:KL}. The approach is quite technique. We recommend \cite[Section 4.6]{TLP} for details; see also
\cite[Section 4]{QY}.


\section{Proofs of Theorem \ref{th:NILF} and Theorem \ref{th:conjugate}}\label{sec:proofs}


Let $g: W\to \mb{C}$ be a complex function defined on a domain $W\subset \mb{C}$. We say $g$ is almost continuous at a point $x\in W$ if for any $\varepsilon>0$,
	\begin{equation}\label{eq:round}
	\lim_{r\to 0^+}\frac{\tu{mes}(\{z\in D(x,r):|g(z)-g(x)|>\varepsilon\})}{\tu{mes}(D(x,r))}=0.
	\end{equation}

It is clear that, if $g$ is almost continuous at $x$, then for any sequence of topological disks $\{U_n\}$ containing $x$ with
$\lim\limits_{n\to \infty}\tu{diam}(U_n)=0$ and $\tu{Shape}(U_n,x)\leq M$ for some constant $M>0$,
The equality \eqref{eq:round} still holds, when replace the round disk $D(x, r)$ by $U_n$.

Combining with Lusin's Theorem and the Lebesgue Density Theorem, one can prove that a measurable function
is almost continuous at almost every point.
\begin{lem}\label{lem:almost}
	A measurable function is almost continuous at each point except a set of measure zero.
\end{lem}
\begin{proof}
	Let $g$ be a complex measurable function defined on a domain $W\subseteq \mathbb{C}$. For any integer $n>0$, by Lusin's Theorem, there exists
	a subset $E_{n}$ of $W$ such that the restriction $g:E_n\to \mb{C}$ is continuous and $\tu{mes}(W \setminus E_{n})<1/n$.
	Let $\widetilde{E}_{n}$ be the set of Lebesgue density points in $E_{n}$. From the Lebesgue Density Theorem,
	$\textrm{mes}(E_{n}\setminus  \widetilde{E}_{n})=0$. By definition, 
	 $g$ is almost continuous at
	each point in $\widetilde{E}_{n}$. Let $E_\infty=\bigcup_{n}\wt{E}_n$. Then
	$\textrm{mes}(W \setminus E_\infty)=0$ and $g$ is almost continuous at
	each point in $E_\infty$.
\end{proof}

Let $\mc{H}(f)$ be the collection of all holomorphic maps $h:U\to V$, where $U$ and $V$ are open sets such that there exist integers $i, j>0$ such that $f^i\circ h=f^j$ on $U$.

The following proposition gives us some conditions on the dynamics of $f$ near a point $z$, under which any $f$-invariant line field $\mu$ either vanishes at the point $z$ or is not almost continuous at $z$.
\begin{prop}\cite[Proposition 3.2]{Shen}\label{prop:shen} Let $f$ be a rational function of degree $\geq 2$ and $z$ be a point in $\mc{J}_f$. If there is a positive constant $C>1$ and a positive integer $N\geq 2$ and a sequence $h_n:U_n\to V_n$ in $\mc{H}(f)$ with the following properties:
	\begin{itemize}
		\item[(1)] $U_n$ and $V_n$ are topological disks containing the point $z$ with
		$$\tu{diam}\,U_n\to 0\tu{ and }\tu{diam}\,V_n\to 0$$ 
		as $n\to \infty$;
		\item[(2)] $h_n$ is a proper map of degree between 2 and $N$;
		\item[(3)] for some $u\in U_n$ such that $h_n'(u)=0$ and for $v=h_n(u)$ we have
		$$\tu{Shape}(U_n, u)\leq C\tu{ and }\tu{Shape}(V_n, v)\leq C;$$
		\item[(4)] $\tu{Shape}(U_n, z)\leq C\tu{ and }\tu{Shape}(V_n, z)\leq C.$
	\end{itemize}
	Then for any $f$-invariant line field $\mu$, either $\mu(z)=0$ or $\mu$ is not almost continuous at the point $z$. 
\end{prop}

Recall that $X_f$ is the set of points $z\in\mc{J}_f$ such that $\omega_{\tu{comb}}(z)$ contains a persistently recurrent critical point.
\begin{lem}\label{lem:density}
	Let $z$ be a wandering point in $\mc{J}_f\setminus X_f$. Then the point $z$ cannot be a Lebesgue density point in $\mc{J}_f$. Moreover, $\bigcap P_n(z)=\{z\}$.
\end{lem}
\begin{proof}
	By a M\"{o}bius conjugacy, one may assume $\mc{J}_f$ is a compact set in $\mb{C}$.
	By definition, it suffices to construct topological disks $E_k$ and $U_k$ such that 
	\begin{enumerate}
		\item $z\in U_k$ and $E_k\subset U_k\cap \mc{F}_f$; recall that $\mc{F}_f$ is the Fatou set of $f$;
		\item for some $x_k\in E_k$ and constant $C$ independent on $k$,
		$$\tu{Shape}(z, U_k)\leq C,~ \tu{Shape}(x_k, E_k)\leq C\tu{ and }\tu{diam}(U_k)\leq C\cdot\tu{diam}(E_k);$$
		\item $\tu{diam}(U_k)\to 0$ as $k\to\infty$.
	\end{enumerate}
	By Corollary \ref{cor:mes}, let $(U, V, n_k)$ be an elevator for the point $z$. Take a round disk $E:=D(x, r_0)\subset U\cap \mc{F}_f$. Consider the map
	$$g_k:=f^{n_k}:(E_k, U_k, V_k)\to (E, U, V),$$
	where $U_k$ and $V_k$ are the components of $f^{-n_k}(U)$ and $f^{-n_k}(V)$ containing the point $z$, respectively, and $E_k$ is a component of $f^{-n_k}(E)$ nested in $U_k$. Let $x_k$ be a point in $E_k\cap f^{-n_k}(x)$. Since the degrees of $g_k$ are uniformly upper bounded by a constant $\delta$, condition (2) follows by Lemma \ref{2:distortionlemma}. To see condition (3), the infinite intersection $\bigcap U_k$ is enclosed by the annuli $H_k:=V_k\setminus \ol{U_k}$. By passing to a subsequence, we may assume $H_k$ are pairwise disjoint. By Lemma \ref{lem:moduli}, we have $$\tu{mod}(H_k)\geq \delta^{-1}\cdot\tu{mod}(V\setminus\ol{U}).$$ The Gr\"{o}tzsch inequality implies  that $\tu{mod}(V_0\setminus\bigcap \ol{U_k})=\infty$. Thus $\bigcap\ol{U_k}$ is a singleton. We have $\bigcap P_n(z)=\bigcap \ol{U_k}=\{z\}$. Since the three conditions are all satisfied for $(E_k, U_k, V_k)$, the proof is complete.
\end{proof}

Recall that in Section \ref{sec:recurrent} for every persistently recurrent critical point $c\in \omega_{\tu{comb}}(z)$, we have constructed a sequence of  well behaved puzzle pieces $(K^-_n, K_n, K_n^+)$ around $c$. Actually one can transfer these puzzle pieces into those that contains the point $z$. That is:
\begin{lem}\label{lem:shape}
	Let $c\in\omega_{\tu{comb}}(z)$ be a persistently recurrent critical point for a point $z\in X_f$. Let $\{(K_n^-, K_n, K_n^+)\}$ be the puzzle pieces constructed in Section \ref{sec:recurrent}. Then 
	\begin{enumerate}
		\item there exists a constant $M=M(d,\kappa)$ such that $\tu{Shape}(K_n, c)\leq M$;
		\item let $l_n>0$ be the first entry time of the point $z$ into $K_n^-$. Let $(V_n^-, V_n, V_n^+)$ be the components of $f^{-l_n}(K_n^-, K_n, K_n^+)$ containing the point $z$, respectively. Then there exists a constant $M_1=M_1(d, \kappa)$ and $M_2=M_2(d, \kappa)$ such that $$\tu{mod}\,(V_n^+\setminus \ol{V_n})\geq M_1\tu{ and }\tu{Shape}(V_n, y)\leq M_2\  \forall y\in V_n^-.$$
		\item $\bigcap P_n(z)=\{z\}$.
	\end{enumerate}

\end{lem}
\begin{proof}
	(1) Since $K_{n-1}^+\sm \ol{K_{n-1}}$ and $K_{n-1}\sm \ol{K_{n-1}^-}$ are disjoint from the set $\mc{PC}(z)$, it follows that $f^{p_n}(c)\in K_{n-1}^-$. Moreover, we have
	$$2\leq \tu{deg}(f^{p_n}|_{A^-_n})=\tu{deg}(f^{p_n}|_{K_n})=\tu{deg}(f^{p_n}|_{A^+_n})\leq D,$$
	where $A^-_n$ and $A^+_n$ are the components of $f^{-p_n}(K_{n-1}^-)$ and $f^{-p_n}(K_{n-1}^+)$ containing the point $c$, respectively. By the Koebe distortion Theorem for univalent function, Lemmas \ref{4:distortionlemma}, \ref{3:distortionlemma} and Theorem \ref{thm:modulus}, we have
	\begin{equation}\label{eq:less}
	\tu{Shape}(K_n, c)\leq C\cdot \tu{Shape}(K_{n-1}, c)^{\frac{1}{2}}
	\end{equation} 
	for some constant $C=M(d, \kappa)$. We repeatedly apply the inequality \eqref{eq:less} as follows
	$$\tu{Shape}(K_n, c)\leq C\cdot C^{\frac{1}{2}}\cdot \tu{Shape}(K_{n-2}, c)^{\frac{1}{2^2}}\leq \cdots\leq C^{1+\frac{1}{2}+\cdots+\frac{1}{2^{n-1}}}\tu{Shape}(K_0, c)^{\frac{1}{2^n}}\leq M$$
	with $M:=C^2\cdot \tu{Shape}(K_0, c)$. The proof of statement (1) is complete.
	
	(2) By Lemma \ref{lem:first-time} and the condition that $$(K_n\sm\ol{K_n^-})\cap\mc{PC}(z)=\es\tu{ and }(K_n^+\sm\ol{K_n})\cap\mc{PC}(z)=\es$$ in Corollary \ref{cor:nest}, we have $\tu{deg}(f^{l_n}|_{V^-_n})=\tu{deg}(f^{l_n}|_{V_n})=\tu{deg}(f^{l_n}|_{V^+_n})\leq d^\kappa$. Then by Lemmas \ref{lem:moduli}, \ref{4:distortionlemma}, \ref{2:distortionlemma}\,(1), \ref{lem:shape}\,(1) and Theorem \ref{thm:modulus}\,(4), statement (2) follows directly. 
	
	(3) It follows by statement (2) and the Gr\"{o}tzsch inequality.
\end{proof}

\subsection{Proof of Theorem~\ref{th:NILF}}

\begin{proof}[Proof of Theorem \ref{th:NILF}]
	We argue by contradiction and assume $\mu$ is an $f$-invariant line field with its support $E:=\tu{supp}(\mu)$ contained in $\mc{J}_f$. Then by Lemma \ref{lem:density} we have
	$$0<\tu{mes}(E)=\tu{mes}(E\cap \wt{X}_f)+\tu{mes}(E\cap(\mc{J}_f\sm\wt{X}_f))=\tu{mes}(E\cap \wt{X}_f),$$
	where $\wt{X}_f\subseteq X_f$ is the set of points in $X_f$ whose orbits are disjoint from $\tu{crit}(f)$. In what follows, we shall show that $\mu$ is not almost continuous at every point $z\in E\cap \wt{X}_f.$ Then by Lemma \ref{lem:almost}, we reach a contradiction.
	
	It suffices to construct maps $h_n: U_n\to V_n$ in $\mc{H}_f$ at the point $z$ such that $h_n$ satisfy the conditions in Proposition \ref{prop:shen}.
	
	Let $(V_n^-, V_n)$ be the puzzle pieces in Lemma \ref{lem:shape}\,(2). 
	Take a minimal and positive integer $v_n$ such that $f^{v_n}(V_n^-)$ covers a critical point, say $c'$. Since $\tu{orb}(z)\cap\tu{crit}(f)=\es$ and $\bigcap P_k(z)=\{z\}$, we have $c'\in\omega\tu{Crit}(z)$. Clearly $v_n\leq l_n$ by Lemma \ref{lem:shape}\,(2). Let $(\Lambda_n^-, \Lambda_n):=f^{v_n}(V^-_n, V_n)$. Note that $f^{v_n}:V_n\to \Lambda_n$ is conformal, as $(K_n\setminus\ol{K_n^-})\cap\mc{PC}(z)=\es$. Let $g_n:\Lambda_n\to V_n$ be the inverse of $f^{v_n}$.
	
	Let $j_n>0$ be the first entry time of $c'$ into $\Lambda_n^-$ and $(\Gamma_n^-, \Gamma_n)$ be the components of $f^{-j_n}(\Lambda_n^-, \Lambda_n)$ containing $c'$, respectively. Similarly, let $u_n>0$ be the first entry time of the point $z$ into $\Gamma_n^-$ and let $(U_n^-, U_n)$ be the components of $f^{-u_n}(\Gamma_n^{-}, \Gamma_n)$ containing the point $z$, respectively. Since $c'\in\omega_{\tu{comb}}(c')$ and $c'\in\omega_{\tu{comb}}(z)$, such numbers $j_n$ and $u_n$ exist.
	
Now consider the composition $h_n:=g_n\circ f^{j_n}\circ f^{u_n}:(U_n^-, U_n)\to (V_n^-, V_n)$. Clearly a point $u$ in the finite set $f^{-u_n}(c')\cap U_n^-$ is a critical point of $h_n$. Moreover, the degree of $h_n$ is less than a constant $\delta=\delta(d,\kappa)$ by Lemma \ref{lem:first-time}. Thus condition (2) of Proposition \ref{prop:shen} holds for $h_n$. By Lemmas \ref{4:distortionlemma} and \ref{2:distortionlemma}, it follows that conditions (3) and (4) of Proposition \ref{prop:shen} hold as well. 

For the shrinking of $\tu{diam}(U_n)$ and $\tu{diam}(V_n)$, note that the depths of puzzle pieces $U_n$ and $V_n$ go to $\infty$ as $n\to\infty$. Then by Lemma \ref{lem:shape}\,(3) it follows.

We have shown that all conditions in Proposition \ref{prop:shen} are satisfied for $h_n$. Hence $\mu$ is not almost continuous at the point $z\in E\cap\wt{X}_f$. The proof is complete.
\end{proof}

\subsection{Proof of Theorem~\ref{th:conjugate}}

	Let $\varphi: W\to \widetilde{W}$ be an orientation-preserving homeomorphism between two open sets in $\mathbb{C}$. For a point $x\in W$,
	let
	$$\underline{H}(\varphi,x)=\liminf_{r\to 0^+}\,\tu{Shape}(\varphi(D(x, r)),\varphi(x))\in [1,\infty].$$
Following \cite{HK}, we say $\varphi$ is a $K$-\emph{quasiconformal mapping} if there exists $1\leq H<\infty$ such that $\underline{H}(\varphi,x)\leq H$ for all $x\in W$, where $K\geq 1$ is a constant depending on $H$.

A topological disk $V\subset \mb{C}$ has {\it $M$-Shape} if there is a point $x\in V$ such that ${\rm Shape}(V,x)\leq M$.

In order to prove Theorem~\ref{th:conjugate}, we need the following QC-Criterion, which is due to Kozlovski-Shen-van Strien:

\begin{lem}\cite[Lemma 12.1]{KSS}\label{l:qc-criterion}
	Let $H>1$, $M>1$ and $m>0$  be constants, $\varphi: W\to \widetilde{W}$ be an orientation-preserving homeomorphism between two Jordan domains $W$ and $\wt{W}$ which extends continuously to the boundary. Let $X_0$ and $X_1$ be disjoint subsets of $W$ such that ${\rm mes}(X_0)={\rm mes}(\varphi(X_0))=0$. Assume that the following hold
	\begin{enumerate}
		\item[(1)] for each $x\in W\setminus (X_0\cup X_1),\,\underline{H}(\varphi,x)<H$;
		\item[(2)] for each $x\in X_0,\,\underline{H}(\varphi,x)<\infty$;
		\item[(3)] there exists a family $\mathcal{V}$ of pairwise disjoint topological disks in $W$ which form a covering of $X_1$, such that for each $V\in \mathcal{V}$, we have
		\begin{itemize}
			\item both $V$ and $\varphi(V)$ have $M$-Shape;
			\item ${\rm mod}(W\setminus \overline{V})>m$ and ${\rm mod}(\widetilde{W}\setminus \overline{\varphi (V)})>m$.
		\end{itemize}
	\end{enumerate}
	Then there exists a $K$-quasiconformal mapping $\psi: W\to \widetilde{W}$ such that $\psi|_{\partial W}=\varphi|_{\partial W}$, where $K>1$ is a constant depending only on $H,\,M$ and $m$.
\end{lem}

\begin{proof}[Proof of Theorem \ref{th:conjugate}]
	Let $f$ and $\wt{f}$ be two non-renormalizable Newton maps conjugated by a homeomorphism $h$. It is clear that $h(\infty)=\infty$, since $\infty$ is the unique repelling fixed point of $f$ and $\wt{f}$. The Fatou sets coincide with the attracting basins of roots, i.e., $\mc{F}_f=B_f$ and $\mc{F}_{\wt{f}}=B_{\wt{f}}$. Then there is a conjugacy which is $K_0$-quasiconformal on $\mc{F}_f$ and coincides with $h$ on $\mc{J}_f$; see \cite{McS}. We denote it still by $h$. Thus there is a constant $H$ such that
	\begin{equation}\label{eq:fatou}
	\underline{H}(h, x)<H\tu{ for each }x\in \mb{C}\setminus\mc{J}_f.
	\end{equation}
	
	Our goal is to prove that the map $h$ is actually global quasiconformal in the complex plane $\mb{C}$.
	
Recall that $X_f$ and $X_{\wt{f}}$ are the defined sets associated to $f$ and $\wt{f}$ in the end of Section \ref{sec:wandering}. Recall also that $\mathcal{P}_n$ is the collection of all puzzle pieces of depth $n$ associated to $f$. Since $h$ is a conjugacy, $h(X_f)=X_{\wt{f}}$ and $\widetilde{\mathcal{P}}_n:=\{\wt{P}:=h(P):P\in \mathcal{P}_n\}$ is the collection of all puzzle pieces of depth $n$ for $\widetilde{f}$. 
	\vskip 0.3cm
	\textbf{Claim } The following statements hold:
	\begin{enumerate}
		\item $\tu{mes}(\mc{J}_f\setminus X_f)=\tu{mes}(\mc{J}_{\wt{f}}\setminus X_{\wt{f}})=0$;
		\item $\underline{H}(h, x)<\infty$ for all $x\in \mc{J}_f\setminus X_f$.
	\end{enumerate}
\begin{proof}[Proof of the claim]
	Statement (1) follows directly by Lemma  \ref{lem:density}.
	
	For (2), assume $x\in \mc{J}_f\setminus X_f$ is wandering. The argument goes similarly in the case that $x$ is preperiodic. By Corollary \ref{cor:mes} there is an elevator $(U, V, n_k)$ for the point $x$. Then we obtain proper mappings 
	$$f^{n_k}: (U_{k}, V_{k})\to (U, V)$$ with $f^{n_k}(x)\in  U\Subset V$ and with their degrees bounded by a constant $\delta=\delta(x)$. Here $U_{k}$ and $V_{k}$ are the components of $f^{-n_k}(U)$ and $f^{-n_k}(V)$ containing the point $x$, respectively. Since $h$ is a conjugacy, by definition $(\wt{U},\wt{V}, n_k)$ is an elevator for the point $\wt{x}:=h(x)$.

	To show $\underline{H}(h, x)<\infty$, it suffices to find a round disk $D_k:=D(x, r_{k})\Subset U_{k}$ for each $k$ with $\wt{D}_{k}:=h(D(x, r_{k}))$ such that $\tu{Shape}(\wt{D}_{k}, \wt{x})\leq M$ for some constant $M=M(x)$ independent on $n$. This is because the diameters of $U_{k}$ tend to zero as $k\to\infty$.
	
	Since the set $S_k:=f^{-n_k}(f^{n_k}(x))\cap V_{k}$ contains at most $\delta$ elements, there is a round disk $D(x, r_{k})\Subset U_{k}$ and a constant $\epsilon_0$ depending only on $\delta$ and $\tu{mod}(V\setminus\ol{U})$, such that
	$$\tu{inf}\,\{\tu{dist}_{\rho_{k}}(\xi, y): \xi\in S_k, y\in\partial D(x, r_k)\}\geq \epsilon_0,$$
	where $\rho_{k}$ is the hyperbolic metric of $V_{k}$. Then by \cite[Lemma 2.2]{CJY} or \cite[Appendix C]{ST}, there is a constant $\epsilon_1$ depending only on $\epsilon_0$ and $\delta$ such that 
	$$\tu{inf}\,\{\tu{dist}_{\rho}(f^{n_k}(x), f^{n_k}(y)): y\in\partial D(x, r_k)\}\geq \epsilon_1,$$
	where $\rho$ is the hyperbolic metric of $V$. Then in sense of the Euclidean metric we have
	$$1/L\leq |f^{n_k}(y)-f^{n_k}(x)|\leq L\ \ \  \forall y\in\partial D(x, r_k)$$
	with some constant $L>1$ independent on $k$. By the continuity of $h$, there is also a constant $\wt{L}>1$ such that
	$$1/\wt{L}\leq |h(f^{n_k}(y))-h(f^{n_k}(x))|\leq \wt{L}\ \ \  \forall y\in\partial D(x, r_k).$$
	Note that $h\circ f^{n_k}(\partial D(x, r_k))=\wt{f}^{n_k}(\partial \wt{D}_{k})$ and it may not be a Jordan curve. Let $B_{k}:=D(\wt{f}^{n_k}(\wt{x}), 1/\wt{L})$. Then $B_{k}\subset \wt{f}^{n_k}(\wt{D}_k)$ and $B_k\cap\wt{f}^{n_k}(\partial \wt{D}_k)=\es$ by the above inequality. 
	
	Consider the map: $\wt{f}^{n_k}:(\wt{U}_{k}, \wt{V}_{k})\to (\wt{U}, \wt{V})$. Let $B_{-k}$ be the component of $\wt{f}^{-n_k}(B_{k})$ containing the point $\wt{x}$. Then $B_{-k}\subseteq \wt{U}_{k}$. Hence the Jordan curve $\partial \wt{D}_{k}$ is essentially contained in the annulus $\wt{U}_{k}\setminus\ol{B_{-k}}$. By Lemma \ref{2:distortionlemma}, there is a constant $C$ independent on $k$ such that $$\tu{Shape}(B_{-k}, \wt{x})\leq C$$ and 
	$$\tu{Shape}(\wt{D}_{k}, \wt{x})\leq 2C \frac{\tu{diam}(\wt{U}_{k})}{\tu{diam}(B_{-k})}\leq 2C^2\frac{\tu{diam}(U)}{\tu{diam}(B_{k})}\leq C^2\cdot \wt{L}\cdot \tu{diam}(U)=:M.$$
	Thus $\underline{H}(h, x)<\infty$. The proof of the claim is complete.	
\end{proof}
	
	Now we fix an integer $n\geq 0$ and a puzzle piece $P_n\in\mc{P}_n$. Consider the homeomorphism $h: P_n\to \wt{P}_n$ with $\wt{P}_n=h(P_n)\in\wt{\mc{P}}_n$. In what follows, we shall check that all the conditions in Lemma \ref{l:qc-criterion} are satisfied for the settings $\varphi:=h|_{P_n}, W:=P_n, \wt{W}:=\wt{P}_n, X_0:=(\mc{J}_f\sm X_f)\cap P_n$ and $X_1:=X_f\cap P_n$.
	
	To see this, first conditions (1) and (2) in Lemma \ref{l:qc-criterion} hold by \eqref{eq:fatou} and the claim. For each $x\in X_1$, take puzzle pieces $(V_{n(x)}, V_{n(x)}^+)$ in Lemma \ref{lem:shape}\,(2) with $n(x)$ sufficiently large such that $V_{n(x)}^+\subset P_n$. Then by Lemma \ref{lem:shape}, we have 
	\begin{equation}\label{eq:moduli}
	\tu{mod}(P_n\setminus\ol{V_{n(x)}})\geq \tu{mod}(V_{n(x)}^+\sm\ol{V_{n(x)}})\geq M_1
	\end{equation}
	and 
	\begin{equation}\label{eq:sh}
	\tu{Shape}(V_{n(x)},x)\leq M_2.
	\end{equation}
	Let $\widetilde{V}_{n(x)}:=h(V_{n(x)})$ and $\widetilde{V}^+_{n(x)}:=h(V_{n(x)}^+)$ be puzzle pieces in the dynamical plane of $\wt{f}$. Since $h$ is a conjugacy, these two puzzle pieces carry the same combinatorial information as $V_{n(x)}$ and $V_{n(x)}^+$. Thus \eqref{eq:moduli} and \eqref{eq:sh} also hold for $\wt{P}_n$, $\wt{V}_{n(x)}$ and $\wt{V}^+_{n(x)}$. 
	
	Among the family 
	$\{V_{n(x)},x\in X_1\}$, let $\mathcal{V}_n$ be the collection of all maximal elements $\{V_1,\cdots,V_i,\cdots\}$ in the sense that 
	$$V_{n(x)}(x)\cap V_i\neq \emptyset\Rightarrow V_{n(x)}(x)\subseteq V_i.$$
	Clearly $\mathcal{V}_n$ consists of pairwise disjoint topological disks which form a covering of $X_1$. Thus condition (3) in Lemma \ref{l:qc-criterion} is satisfied.
	
By applying Lemma \ref{l:qc-criterion} for each $P_n\in\mc{P}_n$ and for each integer $n\geq 0$, we obtain a sequence of $K$-quasiconformal mapping $h_n$ such that $h_n$ and $h$ coincide in $\mb{C}\setminus\bigcup_{P\in\mc{P}_n}P$. Take a subsequence of $h_n$ which converges locally uniformly to a $K$-quasiconformal mapping $\wt{h}$. Since $\bigcap P_n(z)=\{z\}$ for all $z\in\mc{J}_f\setminus\Omega_f$ by Proposition \ref{prop:shrinking}, Lemmas \ref{lem:density} and \ref{lem:shape}, it follows that $\wt{h}=h$. Thus the original mapping $h$ is $K$-quasiconformal. The proof is complete. 
\end{proof}


\begin{thebibliography}{1}
	%
	\frenchspacing
	
	\bibitem[AR]{AR} M. Aspenberg and P. Roesch. {\em Newton maps as matings of cubic polynomials.} Proc. Lond. Math. Soc. (3) \textbf{113} (2016), No.1, 77-112.
	
	\bibitem[BF]{BF14} B. Branner and N. Fagella. {\em Quasiconformal surgery in holomorphic dynamics}. Cambridge University Press, New York, 2014.
	
	\bibitem[BFJK]{BFJK} K. Bara\'{n}ski, N. Fagella, X. Jarque and B. Karpi\'{n}skia.
	{\em On the connectivity of the Julia sets of meromorphic functions}. Invent. math. \textbf{198}(2014), No.3, 591-636.
	
	\bibitem[CJY]{CJY} L. Carleson, P. Jones and J. -C. Yoccoz.
	{\em Julia and John}. Bol. Soc. Brasil Mat.(N.S.),
	{\bf 25} (1994), 1-30.
	
	\bibitem[DS]{DS} K. Drach and D. Schleicher. {\em Rigidity of Newton dynamics}, Adv. Math. {\bf 408} (2022), 108591.
	
	\bibitem[DLSS]{DLSS}K. Drach, R. Lodge, D. Schleicher and M. Sowinski.
	{\em Puzzles and the Fatou-Shishikura injection for Rational Newton maps}. Trans. Amer. Math. Soc. \tb{374}(2021), 2753-2784.
	
	\bibitem[DMRS]{DMRS}K. Drach, Y. Mikulich, J. R\"{u}ckert and D. Schleicher.
	{\em A combinatorial classification of postcritically fixed Newton maps.} Ergod. Th. \&  Dynam. Sys., 39(11)(2019), 2753-2784.
	
	
	\bibitem[GS]{GS} J. Graczyk and S. Smirnov.
	{\em Non uniform hyperbolicity in complex dynamics I, II}. Invent. Math.,
	{\bf 175} (2009), No. 12, 335--415.
	
	\bibitem[Hai]{Hai} P. Ha\"{i}ssinsky.	{\em Ridigity and expansion for rational maps}. J. London Math. Soc.,	{\bf 63} (2001), 128--140.
	
	
	\bibitem[He]{He87} J. Head.
	{\em The combinatorics of Newton's method for cubic polynomials}. PhD thesis, Cornell University, 1987.
	
	\bibitem[HK]{HK} J. Heinonen and P. Koskela. \emph{Definitions of quasiconformality}. Invent. Math. \tb{120} (1995), 61--79.
	
	
	
	\bibitem[KL1] {KL1} J. Kahn and M. Lyubich. {\em The quasi-additivity law in conformal
		geometry}. Ann. of Math., 169 (2009), No. 2, 561--593.
	
	
	\bibitem[KL2] {KL2} J. Kahn and M. Lyubich. {\em Local connectivity of Julia sets for unicritical polynomials}. Ann. of Math., 170 (2009), 413--426.
	
	\bibitem[KS] {KS} O. Kozlovski and S. van Strien.
	{\em Local connectivity and quasiconformal rigidity of non-renormalizable  polynomials}.
	Proceedings of London Math. Soc., {\bf 99}(2009), 275--296.
	
	\bibitem[KSS] {KSS} O. Kozlovski, W. Shen and S. van Strien. {\em Rigidity for
		real polynomials}. Ann. of Math., {\bf 165}(2007), 749--841.
	
	
	\bibitem[LMS1]{LMS1} R. Lodge, Y. Mikulich and D. Schleicher. {\em Combinatorial properties of Newton maps}. arXiv:1510.02761, 2015.
	
	\bibitem[LMS2]{LMS2} R. Lodge, Y. Mikulich and D. Schleicher. {\em A classification of postcritically finite Newton maps}. arXiv: 1510.02771, 2015.
	
	\bibitem[LS]{Levin-vanStrien} G. Levin and S. van Strien. {\em Totally disconnectedness and absence of invariant line fields	for real polynomials}. Asterisque, {\bf 261}(2000), 161--172.
	
	
	
	\bibitem[Mc]{McMbook} C. McMullen.
	{\em Complex Dynamics and Renormalization }.
	Annals of Mathematics Studies {\bf 135}, Princeton University Press, 1994.
	
	
	\bibitem[McS]{McS} C. McMullen and D. Sullivan.
	\emph{Quasiconformal homeomorphisms and dynamics III: The Teichm$\ddot{u}$ller of a holomorphic dynamical system }.
	Adv. Math., {\bf 135}(1998), 351--395.
	
	\bibitem[M1]{M1} J. Milnor.
	{\em Dynamics in one complex variable}.
	Vieweg 1999, 2nd edition 2000.
	
	
	\bibitem[PQRTY] {TLP} W. Peng,  W. Qiu, P. Roesch, L. Tan and Y. Yin. {\em
		A tableau approach of the KSS nest}. Conf. Geom. Dyn.,{\bf 14}(2010), 35--67.
	
	\bibitem[QWY] {QWY} W. Qiu, X. Wang and Y. Yin. {\em Dynamics of McMullen maps,} Adv. Math. \textbf{229} (2012), 2525-2577.
	
	\bibitem[QY] {QY} W. Qiu and Y. Yin.
	{\em Proof of the Branner-Hubbard conjecture on Cantor Julia sets}.
	Science in China, Series A, {\bf 52}(2009), No.1, 45-65.
	
	\bibitem[Roe]{Roe} P. Roesch. {\rm On local connectivity for the Julia set of rational maps: Newton's famous example}. Ann. Math. \textbf{168}(2008), 1-48.
	
	\bibitem[RY]{RY}P. Roesch and Y. Yin. {\em Bounded critical Fatou components are Jordan domains for polynomials}. Sci. China Math. \textbf{65}(2022), 331-358.
	
	\bibitem[RWY]{RWY}P. Roesch, X. Wang and Y. Yin. {\em Moduli space of cubic Newton maps}. Adv. Math. \textbf{322} (2017), 1-59.
	
	\bibitem[Sh1]{Sh08} M. Shishikura.
	{\em Yoccoz puzzle, $\tau$-functions and their appications}. Unpublished, 2008.
	
	\bibitem[Sh2]{Sh09} M. Shishikura.
	{\em The connectivity of the Julia set and fixed points}. in Complex dynamics: families and friends (edited by D. Schleicher), 2009.
	
	\bibitem[ST]{ST} M. Shishikura and Tan Lei, {\em An alternative proof of Ma\~{n}\'{e}'s theorem} on non-expanding Julia sets, in The Mandelbrot set, Theme and Variations, edited by Tan Lei, 265-279, London Math. Soc. Lecture Notes Ser., No.274, Cambrige University Press, Cambrige, 2000.
	
	\bibitem[Shen]{Shen} W. Shen.
	{\em On the measurable dynamics of real rational functions}. Ergodic Theory and Dynamical Systems, {\bf 23}(2003), 957-983.
	
	
	\bibitem[WYZ]{WYZ} X. Wang, Y. Yin and J. Zeng.
	{\em Dynamics of Newton maps}. \emph{Ergodic theory and Dynamical Systems}, published online, 
	\href{https://doi.org/10.1017/etds.2021.168}{DOI: https://doi.org/10.1017/etds.2021.168}.
	
	
	\bibitem[YZ]{YZ} Y. Yin and Y. Zhai.
	{\em No invariant line fields on Cantor Julia sets}. Forum Math., {\bf 22}(2010), No.1, 75-94.
	
	\bibitem[Yo]{Yoccoz} J. C. Yoccoz.
	{\em On the local connectivity of the Mandelbrot set}. Unpublished,
	1990.
	
\end{thebibliography}
\end{document}